\documentclass{amsart}

\newtheorem{fed}{\textbf{Definition}}[section]
\newtheorem{thm}[fed]{\textbf{Theorem}}
\newtheorem{lemma}[fed]{\textbf{Lemma}}
\newtheorem{ex}[fed]{\textbf{Example}}
\newtheorem{rem}[fed]{\textbf{Remark}}
\newtheorem{prop}[fed]{\textbf{Proposition}}
\newtheorem{cor}[fed]{\textbf{Corollary}}
\usepackage{amscd,amssymb,bbm,graphicx,epsfig,psfrag,epic,eepic,latexsym}
\usepackage{amsmath}
\usepackage[arrow,matrix,curve]{xy}
\usepackage{mathrsfs}
\usepackage{color}
\usepackage{verbatim}
\newcommand{\A}{\mathcal{A}}
\newcommand{\N}{\mathbb{N}}

\newcommand{\Q}{\mathbb{Q}}
\newcommand{\Z}{\mathbb{Z}}
\newcommand{\R}{\mathbb{R}}
\newcommand{\C}{\mathbb{C}}

\newcommand{\LL}{\mathscr{L}}
\newcommand{\JJ}{\mathcal{J}}
\newcommand{\FF}{\mathcal{F}}
\newcommand{\p}{\partial}
\newcommand{\wh}{\widehat}
\newcommand{\fk}{\frak{k}}
\newcommand{\eps}{\varepsilon}
\newcommand{\CZ}{{\rm CZ}}
\newcommand{\Id}{{{\mathchoice {\rm 1\mskip-4mu l} {\rm 1\mskip-4mu l}
{\rm 1\mskip-4.5mu l} {\rm 1\mskip-5mu l}}}}

\newcommand*{\mathcolor}{}
\def\mathcolor#1#{\mathcoloraux{#1}}
\newcommand*{\mathcoloraux}[3]{%
  \protect\leavevmode
  \begingroup
    \color#1{#2}#3%
  \endgroup
}

\begin{document}
\title{Symplectic Tate homology}
\author{Peter Albers, Kai Cieliebak, Urs Frauenfelder}
\address{Peter Albers\\
 Mathematisches Institut\\
 Westf\"alische Wilhelms-Universit\"at M\"unster}
\email{peter.albers@wwu.de}

\address{Kai Cieliebak\\
 Mathematisches Institut\\
Universit\"at Augsburg}
\email{kai.cieliebak@math.uni-augsburg.de}

\address{Urs Frauenfelder\\
Department of Mathematics and Research Institute of Mathematics\\
Seoul National University}
\email{frauenf@snu.ac.kr}

\maketitle
\parindent=0pt
\parskip=4pt

\begin{abstract}
For a Liouville domain $W$ satisfying $c_1(W)=0$, we propose in this
note two versions of symplectic Tate homology
$\underrightarrow{H}\underleftarrow{T}(W)$ and
$\underleftarrow{H}\underrightarrow{T}(W)$  
which are related by a canonical map $\kappa \colon \underrightarrow{H}\underleftarrow{T}(W)
\to \underleftarrow{H}\underrightarrow{T}(W)$. Our geometric approach to Tate homology
uses the moduli space of finite energy gradient flow lines of the
Rabinowitz action functional for a circle in the complex plane as a
classifying space for $S^1$-equivariant Tate homology. For rational
coefficients the symplectic Tate homology
$\underrightarrow{H}\underleftarrow{T}(W;\Q)$ has the fixed point
property and is therefore isomorphic to 
$H(W;\mathbb{Q}) \otimes_\Q \mathbb{Q}[u,u^{-1}]$, where
$\mathbb{Q}[u,u^{-1}]$ is the ring of Laurent polynomials over the
rationals. Using a deep 
theorem of Goodwillie, we construct examples of Liouville domains where the canonical map $\kappa$
is not surjective and examples where it is not injective. 
\end{abstract}

\section{Introduction}

Equivariant Tate homology was introduced by Swan \cite{swan} in the
topological context of a finite group $G$ acting on a space $X$, see
the book by Brown \cite{brown}. Later this was generalized to actions
of compact Lie groups by Adem, Cohen and Dwyer \cite{adem_cohen_dwyer}
and Greenlees and May \cite{greenlees-may}, see also Tene \cite{tene}
for another construction using stratifolds. If $G=S^1$ and $X$ is a
finite CW complex, then $S^1$-equivariant Tate homology with rational
coefficients has the fixed point property
$$
\widehat{H}_*^{S^1}(X;\Q)\cong\widehat{H}_*^{S^1}(X^{S^1};\Q)\;,
$$
where $X^{S^1}$ is the fixed point set. For infinite dimensional
spaces this property fails in general, as was observed by Goodwillie
in \cite{goodwillie}. Jones and Petrack \cite{jones-petrack}
introduced a variant of $S^1$-equivariant Tate homology  which retains
the fixed point property, see also Cencelj \cite{cencelj}. 

In this paper we construct analogues of both versions of Tate homology
for symplectic homology, together with a canonical map between the two
versions. Our geometric approach twists symplectic homology 
with Rabinowitz Floer homology on the complex plane. There is
partial overlap between our work and the work of J.\,Zhao \cite{zhao},
who found independently a more algebraic approach to construct the two
symplectic Tate homologies. 

Let us briefly recall the main definition of symplectic homology;
see~\cite{CFH95,FloHof94,Sei08b,Vit99} for details. We restrict to the
case of a Liouville domain, i.e., a compact manifold $W$ with boundary
equipped with a $1$-form $\lambda$ such that $d\lambda$ is symplectic
and $\lambda|_{\p W}$ is a positive contact form. Its completion is
the Liouville manifold $V=W\cup_{\p W}\bigl([1,\infty)\times\p
W\bigr)$ with the $1$-form $\lambda_V$ that equals $\lambda$ on $W$
and $r\lambda$ on $[1,\infty)\times\p W$. The symplectic homology 
$SH_*(W)$ of the Liouville domain $W$ (we will usually omit
$\lambda$ from the notation) is by definition the Floer homology with
$\Z$-coefficients of any Hamiltonian on $V$
that grows quadratically in $r$ at infinity. It is $\Z$-graded 
by Conley-Zehnder indices if the first Chern class $c_1(W)$ (with
respect to any compatible almost complex structure) vanishes, and it
is invariant under Liouville isomorphisms of the completions (i.e.,
diffeomorphisms $f:V\to V'$ such that $f^*\lambda'-\lambda$ is exact
and compactly supported). An $S^1$-equivariant version
$SH^{S^1}_*(W)$ of symplectic homology with respect to the circle
action on the loop space of $V$ has been defined
in~\cite{viterbo,bourgeois-oancea1}. It is a graded module over the
polynomial ring $\Z[u^{-1}]$, where $u$ is a formal variable of degree
$2$. Moreover, it has a localization with respect to the endomorphism
$u^{-1}$, see Section~\ref{loc}. Now we can state our main result. 

\begin{thm}\label{thm:main}
(a) To every Liouville domain $W$ with $c_1(W)=0$ one
can naturally associate two versions of {\em symplectic Tate
  homology}: the {\em Jones-Petrack version}
$\underrightarrow{H}\underleftarrow{T}(W)$ and the {\em
  Goodwillie version} 
$\underleftarrow{H}\underrightarrow{T}(W)$. They are
graded modules over the graded ring $\Z[u,u^{-1}]$ of
Laurent polynomials with coefficients in $\Z$ in a variable
$u$ of degree $2$, and they are invariant under Liouville isomorphisms
of the completions. 

(b) There exists a natural map of $\Z[u,u^{-1}]$-modules
$$
   \kappa \colon \underrightarrow{H}\underleftarrow{T}(W) \to
   \underleftarrow{H}\underrightarrow{T}(W).
$$
(c) With rational coefficients, the Jones-Petrack version 
   $\underrightarrow{H}\underleftarrow{T}$ has the fixed point property
$$
   \underrightarrow{H}\underleftarrow{T}_*(W;\mathbb{Q})
   \cong H_{*+n}(W,\partial W;\mathbb{Q}) \otimes_\Q \mathbb{Q}[u,u^{-1}]. 
$$
(d) With coefficients in a field $\fk$, the Goodwillie version 
$\underleftarrow{H}\underrightarrow{T}(W;\fk)$ agrees with
the localization of the $S^1$-equivariant symplectic homology
$SH^{S^1}(W;\fk)$.  
\end{thm}

{\em Remark. }
The hypothesis $c_1(W)=0$ provides us with well-defined integer
gradings and allows us to work over the ring
$\Z[u,u^{-1}]$. We expect that Theorem~\ref{thm:main} continues to hold
without this hypothesis if we drop the integer grading and replace
$\Z[u,u^{-1}]$ by a suitable Novikov completion. 
\medskip

Let us discuss some examples of symplectic Tate homology groups. 

(1) As an immediate consequence of Theorem~\ref{thm:main} (d),
for a Liouville domain with vanishing equivariant symplectic homology
we have
$$ 
   \underleftarrow{H}\underrightarrow{T}(W;\mathfrak{k}) = 0.
$$ 
For example, this applies to $\C^n$, or more generally to subcritical 
Stein domains~\cite{Cie02a,bourgeois-oancea}. 
More generally, suppose
that a Liouville domain $\widetilde W$ is obtained from a Liouville domain
$W$ by attaching subcritical handles. Then the restriction map induces
an isomorphism $SH^{S^1}(\widetilde W;\fk)\cong SH^{S^1}(W;\fk)$ (this
follows from~\cite{Cie02a}, see also~\cite{bourgeois-oancea2,cieliebak-oancea}). Hence
the Goodwillie versions of symplectic Tate homology satisfy
$$
   \underleftarrow{H}\underrightarrow{T}(\widetilde W;\fk)\cong
   \underleftarrow{H}\underrightarrow{T}(W;\fk). 
$$
Theorem~\ref{thm:main} (c) shows that this fails for the Jones-Petrack
version because the singular homology groups $H_{*+n}(\widetilde W,\partial
\widetilde W;\mathbb{Q})$ and $H_{*+n}(W,\partial W;\mathbb{Q})$ are in general
different. 

(2) The fixed point property for the Jones-Petrack version
$\underrightarrow{H}\underleftarrow{T}$ may fail if the coefficients
are not $\Q$. For example, with integer coefficients we have (see
Proposition~\ref{prop:C}) 
$$
   \underrightarrow{H}\underleftarrow{T}(\C^n) \cong 
   \Q\otimes_\Z\Z[u,u^{-1}].
$$

(3) The Goodwillie version 
$\underleftarrow{H}\underrightarrow{T}(W;\mathbb{Q})$ does in general
not have the fixed point property. This observation is due to
Goodwillie \cite{goodwillie}. For a closed Riemannian manifold $(N,g)$,
denote by $D^*N$ the unit disk cotangent bundle of $N$ with its canonical
structure as a Liouville domain. Goodwillie's theorem implies 

\begin{thm}\label{goodwil}
Assume that $N$ is a closed, simply connected, spin manifold. Then
$\underleftarrow{H}\underrightarrow{T}(D^*N;\mathbb{Q})=\mathbb{Q}[u,u^{-1}].$
\end{thm}

\textbf{Proof: }
The equivariant symplectic homology of $D^*N$ equals the equivariant
homology $H^{S^1}(LN)$ of the free loop space of $N$, see
\cite{abbondandolo-schwarz, salamon-weber, viterbo}. Here the spin
assumption is needed to make coherent orientations work out correctly
\cite{abouzaid, kragh}. By Theorem~\ref{thm:main} (d), we conclude that
$\underleftarrow{H}\underrightarrow{T}(D^*N)$ equals the localization
of $H^{S^1}(LN)$. By Goodwillie's theorem
\cite{goodwillie}, the localization of $H^{S^1}(LN)$ with rational
coefficients only depends on the fundamental group. Since $N$ is simply
connected, we conclude that
$\underleftarrow{H}\underrightarrow{T}(D^*N;\mathbb{Q})$ equals the
localized equivariant loop space homology of a point and the result
follows. \hfill $\square$ 

{\em Remark. }
For symplectic homology with twisted coefficients as
in~\cite{abouzaid}, Theorem~\ref{goodwil} continues to hold without
the spin hypothesis on $N$.

(4) The map $\kappa$ in Theorem~\ref{thm:main} (b) need in general
neither be injective nor surjective, see Example~\ref{ex:kappa}. 
If $\kappa$ has nontrivial kernel, then there exist infinite chains of
periodic orbits connected by Floer cylinders with action going to
infinity; see Section~\ref{ex} for examples of this phenomenon. 
\medskip

The main objective of the current paper is to lay the foundations for
symplectic Tate homology, establish its algebraic properties, and
compute it for some examples. 
However, our interest in this construction is motivated by embedding
questions and dynamical applications. For example, 
Viterbo observed in \cite[page\,1020]{viterbo} that Goodwillie's
theorem implies that the unit disk cotangent bundle $D^*N$ of a simply
connected closed manifold satisfies the ``strong equivariant algebraic
Weinstein conjecture'', namely the fundamental class of $N$ becomes
zero in localized equivariant symplectic homology of $D^*N$ with
rational coefficients. As a consequence, every closed contact type
hypersurface in $D^*N$ carries a closed characteristic. Another
consequence should be finiteness of the Hofer-Zehnder capacity of
$D^*N$, see~\cite{irie} for such an argument in a slightly different
context. Further potential applications concern the construction of
finite energy planes and establishing uniruledness of Liouville
domains in the sense of McLean~\cite{mclean}. 

\paragraph{\bf Acknowledgements.}
Most of this research was carried out during the third author's stay
at M\"unster University on a Humboldt Fellowship, and the first
author's visit to Seoul National University. The first author is
supported by the SFB 878  and the second author by DFG grant CI
45/5-1.

\section{Algebraic preliminaries}\label{sec:alg}

In this section we collect some algebraic facts, most of which can be
found in~\cite{CF-Morse}. For the applications we
have in mind we have to work in the category of graded $\Z$-modules.
For simplicity we will skip the reference to the grading. 

\subsection{Direct and inverse limits}\label{ss:direct}

A {\em direct system} is a tuple $(G,\pi)$
where $G=\{G_a\}_{a \in \R}$ is a family of $\Z$-modules
indexed by the real numbers, and
$\pi=\{\pi_{a_2,a_1}\}_{a_1 \leq a_2}$
is a family of  homomorphisms
$$\pi_{a_2,a_1} \colon G_{a_1} \to G_{a_2}$$
satisfying
$$\pi_{a,a}=\mathrm{id}|_{G_a}, \quad
\pi_{a_3,a_1}=\pi_{a_3,a_2} \circ \pi_{a_2,a_1},\,\,a_1 \leq a_2
\leq a_3.$$
We denote its \emph{inverse limit} as $a\to-\infty$ by
$$\underleftarrow{\lim}\,G = \lim_{\substack{\longleftarrow\\ a}}G_{a}.$$
For $a \in \R$ we have canonical maps
$$\pi_a \colon \underleftarrow{\lim}\,G \to G_a$$
satisfying
$$\pi_{a_2}=\pi_{a_2,a_1} \circ \pi_{a_1},\quad a_1 \leq a_2.$$

In order to make the notation easier
adaptable to our later purposes we denote the direct system by 
$(G,\iota)$ with homomorphisms
$$\iota^{b_2,b_1} \colon G^{b_1} \to G^{b_2}$$
when we talk about direct limits.
We denote its \emph{direct limit} as $b\to\infty$ by
$$\underrightarrow{\lim}\,G=\lim_{\substack{\longrightarrow\\ b}}G^b.$$
For $b \in \R$ we have canonical maps
$$\iota^b \colon G^b \to \underrightarrow{\lim}\,G$$
satisfying
$$\iota^{b_1}=\iota^{b_2} \circ \iota^{b_2,b_1}, \quad b_1 \leq b_2.$$

Direct and inverse limits do not necessarily commute. To describe this, we consider
a double indexed family of $\Z$-modules $G_a^b$ with
$a,b \in \R$. We suppose that
for every $b \in B$
and every $a_1 \leq a_2 \in A$ there exists a homomorphism
$$\pi^b_{a_2, a_1} \colon G_{a_1}^b \to G_{a_2}^b$$
and for every $a \in A$ and $b_1 \leq b_2 \in B$ there exists a
homomorphism
$$\iota_a^{b_2,b_1} \colon G_a^{b_1} \to G_a^{b_2}$$
such that the following
holds. For any fixed $b \in B$ and any fixed $a \in A$ the tuples
$\big(G^b,\pi^b\big)$ and
$\big(G_a,\iota_a\big)$
are direct systems. Moreover, for every
$a_1 \leq a_2$ and $b_1 \leq b_2$ the square
\begin{equation}\label{pii}
\begin{xy}
 \xymatrix{
  G_{a_1}^{b_1} \ar[r]^{\pi^{b_1}_{a_2,a_1}}
  \ar[d]_{\iota^{b_1,b_2}_{a_1}} & G^{b_1}_{a_2}
  \ar[d]^{\iota_{a_2}^{b_2,b_1}}\\
  G^{b_2}_{a_1} \ar[r]_{\pi_{a_2,a_1}^{b_2}} & G^{b_2}_{a_2}
 }
\end{xy}
\end{equation}
is commutative.
We refer to the triple $\big(G,\pi,\iota\big)$ as a
\emph{bidirect system}.
We denote the direct (resp.~inverse) limits for fixed $a$ (resp.~$b$) by
$$\lim_{\longrightarrow}G_{a}=\lim_{\substack{\longrightarrow\\ b}}G_{a}^b, \qquad 
\lim_{\longleftarrow}G^b=\lim_{\substack{\longleftarrow\\ a}}G_{a}^b.$$
Due to the commutation relation between $\pi$ and $\iota$,
for $a_1 \leq a_2 \in A$ and  $b_1 \leq b_2 \in B$ we obtain induced maps
$$\pi_{a_2,a_1} \colon
\lim_{\longrightarrow}G_{a_1}
\to \lim_{\longrightarrow}G_{a_2},
\qquad
\iota^{b_2,b_1} \colon \lim_{\longleftarrow}G^{b_1} \to
\lim_{\longleftarrow}G^{b_2}$$
that make 
$(\underrightarrow{\lim}\,G,\pi)$
and $(\underleftarrow{\lim}\,G,\iota)$ 
direct systems.

\begin{prop}[see~\cite{CF-Morse}]\label{caho}
For a bidirect system $(G,\pi,\iota)$ there exists for every $b \in B$
a unique homomorphism
$$\kappa^b \colon \lim_{\longleftarrow}
G^b \to \lim_{\longleftarrow}
\lim_{\longrightarrow}G$$
and a unique homomorphism
$$\kappa \colon \lim_{\longrightarrow}
\lim_{\longleftarrow}G
\to \lim_{\longleftarrow}
\lim_{\longrightarrow}G$$
such that for each $a \in A$ and $b \in B$ the following
diagram commutes
\[
\begin{xy}
 \xymatrix{
G_a^b  \ar[d]^{\iota^b_a}& \ar[l]_{\pi^b_a}
\underleftarrow{\lim}G^b \ar[r]^{\iota^b}
\ar@{.>}[d]^{\exists!\,\kappa^b}& \underrightarrow{\lim}
\underleftarrow{\lim}G
\ar@{.>}[ld]^/-.8em/{\exists!\, \kappa}\\
\underrightarrow{\lim} G_a  &  \ar[l]_{\pi_a}
\underleftarrow{\lim}\underrightarrow{\lim}G
}
\end{xy}
\]
\end{prop}

In general, the map $\kappa$ is neither 
injective nor surjective; we will see examples of this in 
Section~\ref{ex}.
Conditions under which the canonical homomorphism $\kappa$
is an isomorphism were obtained by B.\,Eckmann and
P.\,Hilton in \cite{eckmann-hilton} and by
A.\,Frei and J.\,Macdonald in \cite{frei-macdonald}. 

\subsection[Bidirect systems of chain complexes]{Bidirect systems
of chain complexes}\label{ss:bidirect}

A {\em bidirect system of chain
complexes} is a quadruple
$(C,p,i,\partial)$,
where $(C,p,i)$ is a bidirect system which in addition is endowed for
each $a \in A$ and $b \in B$ with a boundary operator
$$\partial^b_a \colon C^b_a \to C^b_a$$
which commutes with $i$ and $p$ in the sense that 
\begin{equation}\label{proj1}
p^b_{a_2,a_1} \circ \partial_{a_1}^b=
\partial_{a_2}^b \circ p^b_{a_2,a_1}
\end{equation}
and
\begin{equation}\label{in1}
i_a^{b_2,b_1} \circ \partial_a^{b_1}=\partial_a^{b_2}
\circ i_a^{b_2,b_1}.
\end{equation}
If
$$HC^b_a=\frac{\mathrm{ker}\partial^b_a}{\mathrm{im}\partial^b_a}$$
are the homology groups, and
$Hp^{b}_{a_2,a_1}$ and $Hi^{b_2,b_1}_a$ are the induced maps
on homology, then the triple $(HC,Hp,Hi)$ is a bidirect system. As
in the previous subsection we
let
$$\kappa \colon \underrightarrow{\lim}\underleftarrow{\lim}HC
\to \underleftarrow{\lim}\underrightarrow{\lim}HC$$
be the canonical homomorphism on homology level. We refer to
$$k \colon \underrightarrow{\lim}\underleftarrow{\lim}C
\to \underleftarrow{\lim}\underrightarrow{\lim}C$$
as the canonical homomorphism on chain level. Since
$\partial$ commutes with $i$ and $p$ we obtain an induced map
$$Hk \colon H\big(\underrightarrow{\lim}\underleftarrow{\lim}C\big)
\to H\big(\underleftarrow{\lim}\underrightarrow{\lim}C\big).$$

\begin{prop}\label{ta}
(a) Let $(C,p,i,\partial)$ be a bidirect system of chain complexes of
$\Z$-modules such that all the maps $p^b_{a_2,a_1}$ are
surjective. Then there is a canonical diagram
\begin{equation}\label{maindia1}
\begin{xy}
\xymatrix{
H(\underrightarrow{\lim}\underleftarrow{\lim}C)
\ar[d]^{Hk}
\ar[r]^{\rho}
& \underrightarrow{\lim}
\underleftarrow{\lim}HC
\ar[d]^{\kappa}\\
H(\underleftarrow{\lim}\underrightarrow{\lim}C)   \ar[r]^{\sigma} &
\underleftarrow{\lim}\underrightarrow{\lim}HC
}
\end{xy}
\end{equation}
with $\sigma$ surjective. 

(b) If instead of $\Z$-modules the $C_a^b$ are
graded vector spaces that are finite dimensional in each degree, then
the diagram commutes and $\rho$ is an isomorphism. 
\end{prop}

\textbf{Proof: }
This follows by combining several results in~\cite{CF-Morse}. However,
as everything there is stated in the category of vector spaces, we
need to check which parts carry over to $\Z$-modules. Let us indicate
the required modifications, using the notation from~\cite{CF-Morse}. 

Theorem 3.7 from~\cite{CF-Morse} is proved in~\cite{weibel} for
$\Z$-modules. Using this, it follows that our bidirect system
$(C,p,i,\partial)$ satisfies the ``tameness'' condition
in~\cite{CF-Morse}, except that the map $\nu^b$ is only surjective
rather than an isomorphism. Moreover, the argument from the end of
proof of Theorem A still works to show that $\nu$ is surjective. 
With this, the discussion after Definition 3.3 still yields the
desired diagram and surjectivity of the map $\sigma$. 
If instead of $\Z$-modules the $C_a^b$ are graded vector spaces that
are finite dimensional in each degree, then our bidirect system
$(C,p,i,\partial)$ is actually ``tame'', and Proposition 3.4
from~\cite{CF-Morse} implies that the diagram commutes and $\rho$ is
an isomorphism.   
\hfill $\square$

\begin{rem}
(a) In our applications the spaces $C_a^b$ appear as sublevel sets
  (for $b$) and quotients (for $a$) with respect to two real
  filtrations on a chain complex $(C,\partial)$. In~\cite{CF-Morse} we
  considered the special case that the two filtrations are equal and
  showed that then the map $Hk$ in~\eqref{maindia1} is an
  isomorphism. We will see below that this fails in the case of two
  different filtrations. This failure can be traced to the failure of
  the commutative square~\eqref{pii} to be exact in the sense
  of~\cite{frei-macdonald}. 

(b) The map $\rho$ in Proposition~\ref{ta} is in general not an
  isomorphism for finitely generated $\Z$-modules. A counterexample is
  given in Appendix A of~\cite{CF-Morse}. This failure can be traced
  to the failure of the Mittag-Leffler condition due to the existence
  of infinite chains of subrings such as $\Z\supset 2\Z\supset
  4\Z\supset\cdots$.  
\end{rem}

\begin{rem}\label{rem:short-exact}
Suppose we are given a short exact sequence $0\to A\to B\to C\to 0$ of
bidirect systems of chain complexes of graded vector spaces as in
Proposition~\ref{ta} (b). Since
direct limits preserve exactness, and inverse limits preserve
exactness for finite dimensional vector
spaces~\cite{eilenberg-steenrod}, we can apply either first the
inverse limit, then the direct limit and then homology, or first
homology, then the inverse limit and then the direct limit to obtain
long exact sequences fitting in the following commuting diagram with
the map $\rho$ from~\eqref{maindia1}:
\begin{equation}
\begin{xy}
\xymatrix{
\cdots \ar[r] &
H(\underrightarrow{\lim}\underleftarrow{\lim}A)
\ar[d]^{\rho_A}_\cong \ar[r] & 
H(\underrightarrow{\lim}\underleftarrow{\lim}B)
\ar[d]^{\rho_B}_\cong \ar[r] & 
H(\underrightarrow{\lim}\underleftarrow{\lim}C)
\ar[d]^{\rho_C}_\cong \ar[r] & \cdots
\\
\cdots \ar[r] &
\underrightarrow{\lim}\underleftarrow{\lim}\text{HA}
\ar[r] &
\underrightarrow{\lim}\underleftarrow{\lim}\text{HB}
\ar[r] &
\underrightarrow{\lim}\underleftarrow{\lim}\text{HC}
\ar[r] & \cdots
}
\end{xy}
\end{equation}
\end{rem}

\subsection{Module structure over Laurent polynomials}

In our applications we will have for all $a,b$ additional isomorphisms
$$u_a^b:C_a^b\overset{\cong}{\longrightarrow} C_{a+1}^b$$
commuting with all the other maps in the obvious sense. 
By functoriality of direct and inverse limits and homology, the
$u_a^b$ induce isomorphisms on the four groups in~\eqref{maindia1}
commuting with the maps in this diagram. Denoting these isomorphisms
simply by $u$, we thus have 

\begin{cor}\label{cor:module}
Additional isomorphisms $u_a^b$ as above make the four groups
in~\eqref{maindia1} modules over the ring $\fk[u,u^{-1}]$ of Laurent
polynomials in $u$ and the maps in this diagram module
homomorphisms. \hfill $\square$ 
\end{cor}

\begin{rem}
(a) In our applications the $u_a^b$ will shift gradings by $2$. 

(b) For a field $\fk$, the ring $\fk[u,u^{-1}]$ is a principal ideal
  domain. To see this, let $I$ be an ideal in $\fk[u,u^{-1}]$. Then
  $I\cap  \fk[u]$ is an ideal in the principal ideal domain  $\fk[u]$
  and thus generated by one polynomial $p$. Now for any $q\in I$  we
  have $u^nq\in I\cap \fk[u]$ for some large integer $n$, so $u^nq=rp$
  for some $r\in\fk[u]$ and thus $q=u^{-n}rp$. 

(c) By~\cite[Section 3.9, Exercise 3]{jacobson}, modules over a
  principal ideal domain have a well-defined rank satisfying the rank
  formula for a submodule $N\subset M$: 
$${\rm rank}\,M = {\rm rank}\,N + {\rm rank}\,M/N.$$
\end{rem}

\subsection{Localization}\label{loc}
In this subsection we work in the category of vector spaces. 
Assume that $V$ is a vector space over a field $\frak{k}$ and $T \colon V \to V$
is a linear map. We define $V^T \subset \prod_{i\in\N}V$ to be the
subvector space consisting of $(v_i)_{i \in \mathbb{N}}$ satisfying $Tv_{i+1}=v_i$. 
If one thinks of the pair $(V,T)$ as a direct system with $V^k=V$ for $k \in \mathbb{N}$ and
$T^k=T \colon V^k \to V^{k-1}$ for $k \geq 2$, then we can think of $V^T$ as the inverse limit
\begin{equation}\label{eq:VT}
V^T=\lim_{\substack{\longleftarrow\\ k}}V^k.
\end{equation}
(Note that here the notation differs from Section~\ref{ss:direct} in
that $k$ tends to $+\infty$ rather than $-\infty$.)
The map $T$ induces a linear map
$\overline{T} \colon V^T \to V^T$ given by
$$\overline{T}(v_i)_{i \in \mathbb{N}}=(Tv_i)_{i \in \mathbb{N}}.$$
Although $T$ was not assumed to be invertible, the map $\overline{T}$ is invertible with inverse given by
$$\overline{T}^{-1} \colon V^T \to V^T, \quad (v_i)_{i \in \mathbb{N}} \mapsto
(v_{i+1})_{i \in \mathbb{N}}.$$
In particular, $V^T$ becomes a module over the ring of Laurent
polynomials $\mathfrak{k}[u,u^{-1}]$ where $u$ acts via $\overline{T}$
on $V^T$. We call the $\mathfrak{k}[u,u^{-1}]$-module $V^T$ the {\em
  localization} of $V$ with respect to $T$. 

There is a natural map
$$P=P_T \colon V^T \to V, \quad (v_i)_{i \in \mathbb{N}} \mapsto v_1.$$
Note that $P$ interchanges the maps $\overline{T}$ and $T$
$$TP=P\overline{T}.$$ 
\begin{lemma}\label{lem:P-iso}
If $T$ is an isomorphism, then $P$ is an isomorphism as well.
\end{lemma}
\textbf{Proof: }We first show that $P$ is injective. To see this,
suppose that $(v_i)_{i \in \mathbb{N}}$ is in the kernel of $P$. This
means that $v_1=0$. By induction on the formula $v_i=Tv_{i+1}$ and
using the injectivity of $T$, we conclude that $v_i=0$ for every $i \in
\mathbb{N}$. This shows that $P$ is injective. To see that $P$ is
surjective, let $v \in V$. Since $T$ is an isomorphism, the 
element $(T^{-i+1}v)_{i \in \mathbb{N}}$ exists in $V^T$. But
$P(T^{-i+1}v)_{i \in \mathbb{N}}=v$. This proves surjectivity
of $P$, and hence the lemma. \hfill $\square$
\\ \\
In particular, Lemma~\ref{lem:P-iso} implies that if $T$ is an
isomorphism, then the pair $(V,T)$ is naturally identified with the
pair $(V^T,\overline{T})$ via the map $P$.  

\begin{lemma}\label{lem:P-inj}
Assume that $V$ is finite dimensional. Then $P$ is injective.
\end{lemma}
\textbf{Proof: }Assume that $v=(v_i)_{i \in \mathbb{N}} \in V^T$
lies in the kernel of $P$, i.e., $v_1=0$. We suppose by contradiction that $v \neq 0$. Hence there exists $i \in \mathbb{N}$ satisfying $v_i \neq 0$. Let $i_0>1$ be the minimal positive integer with this property. Hence for every  
$i \geq i_0$ the vector $v_i$ has the property that $T^jv_i \neq 0$
for $0 \leq j \leq i-i_0$ but $T^{i-i_0+1}v_i=0$. This implies that
the vectors $v_i$ for $i \geq i_0$ are linearly independent. But this
contradicts the assumption that $V$ is finite dimensional and the
lemma is proved. \hfill $\square$

\begin{ex}{\rm 
In infinite dimensions, Lemma~\ref{lem:P-inj} is far from
true. Consider for example a vector space $V$ of infinite countable
dimension with basis $\{e_i\}_{i \in \mathbb{N}}$. Let $T \colon V \to
V$ be the shift operator 
$$T e_i:=\left\{\begin{array}{cc}
e_{i-1} & i>1 \\
0 & i=1.
\end{array}\right.
$$
A basis for $V^T$ is given by $\overline{e}_j$, $j \in \mathbb{Z}$, defined by
$$(\overline{e}_j)_i:=\left\{\begin{array}{cc}
e_{i+j} & i+j \geq 1\\
0 & i+j \leq 0.
\end{array}\right.$$
The induced linear operator becomes the shift operator
$$\overline{T}\overline{e}_j=\overline{e}_{j-1}, \quad j \in \mathbb{Z}.$$
The kernel of $P$ is spanned by the vectors $\overline{e}_j$
for $j<0$, so $P$ is not injective.\hfill $\square$}
\end{ex}

On $V^T$ we define a descending filtration as follows. If $v=(v_i)_{i \in \mathbb{N}}$ we set
$$|v|:=\min\{ i \in \mathbb{N}: v_i \neq 0\}$$
with the convention that the minimum of the empty set equals $\infty$, i.e., $|0|=\infty$. For $k \in \mathbb{N}$ we set 
$$Z^T_k := \{v \in V^T\;\bigl|\; |v| > k\} = \{v \in V^T\mid v_1=\cdots =v_k=0\}.$$ 
Note that $Z^T_k$ is $\overline{T}$-invariant. Therefore if 
$$V^T_k:=V^T/Z^T_k$$
is the quotient vector space, then $\overline{T}$ induces a map
$$T_k \colon V^T_k \to V^T_k.$$
There is a well defined map
$$Q_k \colon V^T_k \to V$$
which is given for $[v]=[(v_i)_{i \in \mathbb{N}}] \in V^T_k$ by
$$Q_k[v]:=v_k.$$
Note that $Q_k$ is injective and satisfies
$$T Q_k=Q_k T_k.$$
We next describe its image. For this purpose set
$$V_T:=\bigcap_{j \in \mathbb{N}} T^j V \subset V.$$
Observe that $V_T$ is a $T$-invariant subspace of $V$ on which $T$ acts surjectively. In fact, $V_T$ is the largest $T$-invariant subspace of $V$ on which $T$ acts surjectively. In particular, $V_T=V$ iff $T$ is surjective. Moreover,
$$\mathrm{im}\,Q_k=V_T.$$
It is worth to mention that the image is independent of $k \in \mathbb{N}$.

Since $Z_k^T \subset Z^T_{k-1}$, we get projection maps
$$\pi_k \colon V^T_k \to V^T_{k-1}$$
and hence a direct system of vector spaces $(V^T, \pi)$.
If $v=[(v_i)_{i \in \mathbb{N}}] \in V^T_k$, we compute
$$Q_{k-1} \pi_k[v]=v_{k-1}=T v_k=T Q_k [v],$$
so the following diagram commutes:
\begin{equation*}
\begin{xy}
\xymatrix{
V_k^T
\ar[d]^{Q_k}_\cong
\ar[r]^{\pi_k}
& V_{k-1}^T
\ar[d]^{Q_{k-1}}_\cong\\
V_T
\ar[r]^{T} &
V_T\,.
}
\end{xy}
\end{equation*}
We denote by $V_T^T:=(V_T)^T$ the localization of $V_T$ with respect
to the map $T|_{V_T}$, defined by the inverse limit~\eqref{eq:VT}. So
the preceding commuting diagram yields

\begin{lemma}\label{vtt1}
The maps $Q_k$ give rise to an isomorphism  
$$Q: \underleftarrow{\lim}V^T_k \to V_T^T$$
which interchanges $\underleftarrow{\lim}T_k$ and $\overline{T}$. In particular, the
spaces $\underleftarrow{\lim}V^T_k$ and $V_T^T$ are isomorphic as
$\mathfrak{k}[u,u^{-1}]$-modules. \hfill$\square$
\end{lemma}
We further have
\begin{equation}\label{vtt}
V^T_T=V^T.
\end{equation}
Indeed, $V^T_T \subset V^T$ is clear. On the other hand, if
$v=(v_i)_{i \in \mathbb{N}} \in V^T$, then the condition $T
v_{i+1}=v_i$ implies that $v_i \in \bigcap_{j=1}^\infty T^j V=V_T$ for
every $i \in \mathbb{N}$ and therefore $v \in V^T_T$. This proves
(\ref{vtt}). We point out that if $T$ is surjective, then (\ref{vtt})
is actually trivial because in this case we already have $V_T=V$.

As a consequence of Lemma~\ref{vtt1} and equation (\ref{vtt}) we get
\begin{cor}
The  spaces $\underleftarrow{\lim}V^T_k$ and $V^T$ are isomorphic as $\mathfrak{k}[u,u^{-1}]$-modules.
\end{cor}

Next, we will carry over the preceding discussion from vector spaces
to chain complexes. 

\begin{fed}\label{tatetriple}
A \emph{Tate triple} $(V,T,\partial)$ consists of
a vector space $V$, a linear map $T \colon V \to V$,
and a boundary operator $\partial \colon V \to V$ commuting
with $T$.
\end{fed}

Let $(V,T,\partial)$ be a Tate triple. We define a boundary operator
$\overline{\partial}$ on $V^T$ by setting for $(v_i)_{i \in
  \mathbb{N}} \in V^T$ 
$$\overline{\partial}(v_i)_{i \in \mathbb{N}}:=(\partial v_i)_{i \in \mathbb{N}}.$$
Note that $\overline{\partial}$ commutes with $\overline{T}$ 
and for every $k \in \mathbb{N}$ the subspace $V^T_k$ is invariant under $\overline{\partial}$. 
Therefore, $\overline{\partial}$ induces boundary operators
$$\partial_k \colon V^T_k \to V^T_k$$
which commute with the maps $T_k$.
Note further that $V_T$ is invariant under $\partial$ and the maps $Q_k$ identify the triples
$(V^T_k,T_k,\partial_k)$ and $(V_T,T,\partial)$. In particular, the map $Q_k$ induces an isomorphism
on homology
$$HQ_k \colon H(V^T_k, \partial_k) \to H(V_T,\partial)$$
which interchanges the induced maps
$$HT_k \colon H(V^T_k, \partial_k) \to H(V^T_k,\partial_k), \quad
HT \colon H(V_T,\partial) \to H(V_T,\partial).$$
Finally, we observe that the maps $\pi_k \colon V^T_k \to V^T_{k-1}$
interchange the boundary operators $\partial_k$ and $\partial_{k-1}$,
so that we obtain a direct system $\big(H(V_k^T,\partial_k), H
\pi_k\big)$. So we have proved the following generalization of
Lemma~\ref{vtt1}, where $H(V_T,\partial)^{HT}$ denotes the
localization of $H(V_T,\partial)$ with respect to the map $HT$.  

\begin{lemma}\label{hvt}
Assume that $(V,T,\partial)$ is a Tate triple. Then the maps $HQ_k$ give rise to an isomorphism
$$HQ \colon \underleftarrow{\lim}H(V_k^T,\partial_k) \to H(V_T,\partial)^{HT}$$
which interchanges $\underleftarrow{\lim} HT_k$ and $\overline{HT}$. In particular, the spaces
$\underleftarrow{\lim}H(V_k^T,\partial_k)$ and $H(V_T,\partial)^{HT}$ are isomorphic as
$\mathfrak{k}[u,u^{-1}]$-modules. \hfill$\square$
\end{lemma}

Recall that if $T$ is surjective, then $V_T=V$. Therefore, we obtain
\begin{cor}\label{locmain}
Assume that $(V,T,\partial)$ is a Tate triple and $T \colon V \to V$ is surjective. Then the spaces
$\underleftarrow{\lim}H(V_k^T,\partial_k)$ and $H(V,\partial)^{HT}$ are isomorphic as
$\mathfrak{k}[u,u^{-1}]$-modules. \hfill $\square$
\end{cor}

Without the surjectivity assumption the assertion of Corollary~\ref{locmain} can
fail, despite the fact that $V^T_T=V^T$ by (\ref{vtt}). The following
example describes such a scenario.

\begin{ex}{\rm
Let $V$ be a vector space with basis vectors
$\{e_{i,j}, f_{i,j}\mid i,j \in \mathbb{N}, j \geq i\}$. Define $T$ and $\partial$ on basis vectors by
$$T e_{i,j}:=\left\{\begin{array}{cc}
e_{i-1,j} & i \geq 2\\
0 & i=1,
\end{array}\right.,\quad T f_{i,j}:=\left\{\begin{array}{cc}
f_{i-1,j} & i \geq 2\\
0 & i=1,
\end{array}\right. $$
and
$$\partial f_{i,j}:=e_{i,j}+e_{i,j+1}, \quad \partial e_{i,j}:=0.$$
Note that $\partial$ is a boundary operator which commutes with $T$, so that $(V,T,\partial)$
is a Tate triple. However, the map $T$ is not surjective. We claim that 
\begin{equation}\label{triv}
V_T=\{0\}.
\end{equation}
To see this, we first decompose $V$ as follows. Define subspaces
$$E:=\big\langle e_{i,j}\mid i,j \in \mathbb{N}, j \geq i\big\rangle \subset V, \quad
F:=\big\langle f_{i,j}\mid i,j \in \mathbb{N}, j \geq i\big\rangle \subset V.$$
Note that
$$V=E \oplus F$$
and both subspaces are $T$-invariant. We therefore have
$$V_T=E_T \oplus F_T.$$
Moreover, observe that the linear map $\Phi \colon E \to F$ which is given on basis vectors
by $\Phi(e_{i,j})=f_{i,j}$ is a $T$-equivariant isomorphism between $E$ and $F$. Therefore,
$E_T$ is isomorphic to $F_T$ 
and we are left with showing that $E_T=\{0\}$. Indeed, for every $k \in \mathbb{N}$ we have
$$T^k E=\big\langle e_{i,j}\mid i,j \in \mathbb{N}, j \geq i+k\big \rangle$$
and therefore
$$E_T=\bigcup_{k \in \mathbb{N}}T^k E=\{0\}.$$
This finishes the proof of (\ref{triv}).
In view of Lemma~\ref{hvt} we deduce from (\ref{triv}) that
$$\underleftarrow{\lim}H(V_k^T,\partial_k)=\{0\}.$$
To compute $H(V,\partial)^{HT}$, we first describe $H(V,\partial)$. Note that 
for every $i \in \mathbb{N}$ the vector $e_{i,i}$ gives rise to a nontrivial homology class
$$\epsilon_i:=[e_{i,i}]$$
which coincides with $[e_{i,j}]$ for every $j \geq i$. The homology becomes
$$H(V,\partial)=\big \langle \epsilon_i\mid i \in \mathbb{N}\rangle,$$
and the induced map becomes the shift operator
$$HT \epsilon_i=\left\{\begin{array}{cc}
\epsilon_{i-1} & i>1 \\
0 & i=1.
\end{array}\right.
$$
We conclude that $H(V,\partial)^{HT}$ is isomorphic to the free
$\mathfrak{k}[u,u^{-1}]$-module on one generator.
In particular, $\underleftarrow{\lim}H(V_k^T,\partial_k)$ and
$H(V,\partial)^{HT}$ are not isomorphic. 
\hfill $\square$}
\end{ex}

{\em Remark. } We say that a Tate triple $(V,T,\partial)$ is \emph{graded of degree $d$}  if the vector space $V$ is additionally graded, $T \colon V \to V$ is a map of degree $\deg(T)=d$, and the boundary operator is a map of degree $\deg(\partial)=-1$. We can define a grading on $V^T$ as well by setting for a nonzero
$v=(v_i)_{i \in \mathbb{N}}$
$$\deg(v)=\deg(v_i)+i\,d,$$
for an arbitrary $i \in \mathbb{N}$ satisfying $v_i \neq 0$.
Note that this well defined, i.e., independent of the choice of $i$. In this setup, Lemma~\ref{hvt} is true
in the graded sense, where the ring of Laurent polynomials
$\mathfrak{k}[u,u^{-1}]$ also has 
to be graded with $\mathrm{deg}(u)=d$. \hfill $\square$

\section{$S^1$-equivariant Tate homology}

\subsection{Rabinowitz Floer homology on $\C$}\label{ss:Rab}

To define $S^1$-equivariant Tate homology via the Borel construction,
we look for a space with a free $S^1$-action and an invariant Morse
function whose Morse homology vanishes in all degrees. Such a space
can be described as follows. Consider the Hilbert space 
$$
   \LL_\C := \{z=(z_k)_{k\in\Z}\;\bigl|\;
   z_k\in\C,\ \sum_{k\in\Z}(1+k^2)|z_k|^2<\infty\} 
$$
with the Morse function $\A:\LL_\C\to\R$,
$$
   \A(z) := \pi\sum_{k\in\Z}k|z_k|^2. 
$$
(The reason for the normalization constant $\pi$ will become clear in
   a moment.) The action of the circle $S^1=\R/\Z$ on $\LL_\C$ via 
$$
   (\tau\cdot z)_k := e^{2\pi i\tau}z_k
$$
leaves $\A$ invariant and is free on the infinite dimensional sphere
$$
   S^\infty_\infty := \{z\in\LL_\C\;\bigl|\; \|z\|^2:=\sum_{k\in\Z}|z_k|^2=1\}\subset\LL_\C.
$$
The function $\A$ descends to a Morse function on the quotient
$\C P^\infty_\infty=S^\infty_\infty/S^1$ whose critical points $[z^{(\ell)}]$, $\ell\in\Z$, are
given by 
$$
   |z^{(\ell)}_\ell|=1,\quad z^{(\ell)}_k=0 \text{ for }k\neq\ell. 
$$ 
Note that the critical points have infinite index and coindex and
critical values $\A(z^{(\ell)})=\pi\ell$. 
The space $\LL_\C$ also carries a natural 
$\Z$-action given by the shifts
$$
   (n*z)_k := z_{k-n}. 
$$
This action preserves $S^\infty_\infty$, commutes with the $S^1$-action, and
satisfies
$$
   \A(n*z) = \A(z)+n\pi\|z\|^2. 
$$
In particular, $\Z$ acts on the critical points on $S^\infty_\infty$ by
$$
   n*z^{(\ell)} = z^{(\ell+n)}. 
$$
Identifying elements of $\LL_\C$ with Fourier series 
$$
   z(t) = \sum_{k\in\Z}z_ke^{2\pi ikt},
$$
we see that $\LL_\C$ corresponds to the Sobolev space $W^{1,2}(S^1,\C)$
and $\|\ \|$ to the $L^2$-norm 
$$
   \|z\|^2 = \int_0^1|z(t)|^2dt. 
$$
The function $\A$ is the classical symplectic action
$$
   \A(z) = \int_0^1z^*\lambda_\C = -\frac{1}{2}\int_0^1{\rm Im}(z\dot{\bar
     z})dt,\qquad \lambda_\C = \frac{1}{2}(x\,dy-y\,dx),
$$
and the actions of $S^1$ and $\Z$ are given by
$$
   \tau\cdot z(t) = e^{2\pi i\tau}z(t),\qquad n*z(t) = e^{2\pi
     int}z(t). 
$$
Note that the circle action on $\LL_\C$ is induced by the circle
action on the target $\C$ of the loop space and {\em not} on the
domain $S^1$. The reason is that this action is free on the sphere
$S^\infty_\infty$, while the action coming from circle action on the domain
is not. 

The restriction of $\A$ to the sphere $S^\infty_\infty$ can be conveniently
described in terms of the Rabinowitz action functional for the unit
circle in $\mathbb{C}$, 
$$
   \mathcal{A}^\mu \colon \mathscr{L}_{\mathbb{C}} \times \mathbb{R}
   \to \mathbb{R},
$$
$$
   \mathcal{A}^\mu(z,\eta) := \int
   z^*\lambda_{\mathbb{C}}-\eta\int\mu(z)dt.
$$
Here $\eta\in\R$ is a Lagrange multiplier and 
$$
   \mu \colon \mathbb{C} \to \mathbb{R}, \qquad z \mapsto\pi(|z|^2-1)
$$
is the moment map for the standard circle action $(t,z) \mapsto
e^{2\pi it}z$ on $\mathbb{C}$. The critical points of $\A^\mu$ appear
in critical circles obtained by applying the circle action to the
pairs 
$$
   w^{(\ell)}:=(z^{(\ell)},\ell),\qquad z^{(\ell)}(t)=e^{2\pi i\ell t}.
$$
They have actions $\A^\mu(z^{(\ell)},\ell)=\pi\ell$. The Rabinowitz
action functional is invariant under the circle action
$\tau\cdot(z,\eta)=(\tau\cdot z,\eta)$, and with respect to the
$\Z$-action 
$$
   n*(z,\eta) := (n*z,\eta+n)
$$
it satisfies
$$
   \A^\mu(n*z,\eta+n) = \A^\mu(z,\eta)+\pi n. 
$$
The gradient of $\A^\mu$ with respect to the $L^2$-metric on
$\LL_\C$ and the standard metric on $\R$ is given by
$$
   \nabla\A^\mu(z,\eta) = \Bigl(-i\dot z-2\pi\eta
   z,-\pi(\|z\|^2-1)\Bigr). 
$$
Thus (positive) gradient flow lines of $\A^\mu$ are maps
$(z,\eta):\R\to\LL_\C\times\R$ whose Fourier coefficients satisfy the
following system of ordinary differential equations, where $'$ denotes
the derivative with respect to $s\in\R$: 
\begin{equation}\label{eq:Fourier}
\begin{cases}
   z_k' = 2\pi(k-\eta)z_k \\
   \eta' = -\pi(\|z\|^2-1). 
\end{cases}
\end{equation}
Note that the subspaces where some of the $z_k$ are zero are invariant
under the gradient flow. In particular, gradient flow lines connecting
two critical points never pass through the constant loop $z\equiv 0$,
so they remain in the region $\LL_\C^*:=\LL_\C\setminus\{0\}$ where the
$S^1$-action is free and can be 
projected to $S^\infty_\infty\times\R$ by normalization. Since both the
action functional and the metric are $S^1$-invariant, gradient flow
lines descend to the quotient by $S^1$. Moreover, since the
$\Z$-action preserves the metric and changes the action functional
only by additive constants, we get a $\mathbb{Z}$-action on 
gradient flow lines which we will describe next. 


Note first that for integers $m\leq n$ the intersection of the stable
manifold (with respect to the negative gradient flow) of $[w^{(m)}]$ and
the unstable manifold of $[w^{(n)}]$ projects onto the
$2(n-m)$-dimensional projective subspace
$$
   \C P_m^n := \{[z]\;\bigl|\; z_k=0 \text{ for all }k<m \text{ and for all
   }k>n\}\subset\C P^\infty_\infty. 
$$
Moreover, according to~\cite[Proposition A.2]{frauenfelder}, the indices of a
Lagrange multiplier functional differ from the indices of the function
restricted to the constraint hypersurface only by a global shift. 
Hence the critical points $[w^{(m)}]$ and $[w^{(n)}]$ have finite
index difference $2(n-m)$. Note that the $\C P_m^n$ are closed
submanifolds of $\C P^\infty_\infty$ satisfying
$$
   \C P_m^n\cap\C P_{m'}^{n'} = \C P_{\max(m,m')}^{\min(n,n')}, 
$$
and the action by $\ell\in\Z$ maps $\C P_m^n$ onto $\C P_{m+\ell}^{n+\ell}$. 

The {\em equivariant Rabinowitz Floer complex on $\C$} (with
$\Z$-coefficients) is the Floer chain complex of the functional
$\A^\mu$ on the quotient space $(\LL_\C^*\times\R)/S^1$: Its chain
group $FC^{S^1}$ is the free $\Z$-module generated by the critical points
$[w^{(n)}]$, $n\in\Z$, and its boundary operator $\p^{S^1}$ counts
negative gradient lines between critical points of index difference
one. Let us write the group ring of $\Z$ 
as the ring 
$$
   \Lambda := \Z[u,u^{-1}],\qquad |u|=2
$$
of Laurent polynomials in a formal variable $u$ of degree $2$. The
$\Z$-action gives an isomorphism between $\Lambda$ and the chain group
$FC^{S^1}$ by identifying $u^n$ with the critical
point $w^{(n)}$ for $n\in\Z$. Since the critical points $[w^{(n)}]$ and
$[w^{(m)}]$ on the quotient $S^\infty_\infty/S^1=\C P^\infty_\infty$ have index
difference $2(n-m)$, this isomorphisms fixes gradings of the
critical points compatible with the relative gradings.
Since all index differences are even, the boundary operator vanishes
and the equivariant Rabinowitz Floer homology equals $\Lambda$. 

On the other hand, since the unit circle in $\C$ is displaceable by a
Hamiltonian diffeomorphism, it follows
from~\cite{cieliebak-frauenfelder} that the non-equivariant Rabinowitz
Floer homology vanishes. So we have shown

\begin{lemma}
The non-equivariant Floer homology $FH(A^\mu)$ vanishes. The
equivariant Floer homology $FH^{S^1}(\A^\mu)$ equals $\Lambda =
\Z[u,u^{-1}]$, the ring of Laurent polynomials in a formal variable $u$
of degree $2$. \hfill$\square$
\end{lemma}

Let us describe the non-equivariant Rabinowitz Floer complex more
explicitly. We fix a Morse function on the critical circle
corresponding to $w^{(0)}$ with two critical points, the maximum
$w^{(0)}_+$ and the minimum $w^{(0)}_-$. Via the $\Z$-action, we
obtain Morse functions on the critical circles corresponding to
$w^{(n)}$ with maximum $w^{(n)}_+=u^nw^{(0)}_+$ and minimum
$w^{(n)}_+=u^nw^{(0)}_+$. Their indices are
$$
   |w^{(n)}_+| = 2n+1,\qquad |w^{(n)}_-| = 2n,
$$ 
and vanishing of non-equivariant Rabinowitz Floer homology implies
that 
$$
   \p w^{(n)}_+ = 0,\qquad \p w^{(n)}_- = \pm w^{(n-1)}_+.
$$
Since the boundary operator decreases the action, we have for each
$b\in\R$ a subcomplex $FC^b\subset FC$ generated by critical points of
action $\leq b$. For $a\leq b$ we denote by $FC_a^b$ the quotient of
$FC^b$ by the subcomplex generated by critical points of action
$<a$, and we set $FC_a:=FC_a^\infty$. We denote the corresponding
homology groups by $FH^b$, $FH_a^b$, and $FH_a$. The same definitions
apply in the equivariant case. Then the explicit description of the
chain complexes above yields the filtered Rabinowitz Floer homology
groups for integers $m\leq n$: 
\begin{gather*}
   FH_m^n(\A^\mu) = \Z w_+^{(n)}\oplus \Z w_-^{(m)},\cr
   FH^n(\A^\mu) = \Z w_+^{(n)},\qquad
   FH_m(\A^\mu) = \Z w_-^{(m)},\cr
   FH_m^{S^1,n}(\A^\mu) = u^m\Z[u]/u^{n+1}\Z[u],\cr
   FH^{S^1,n}(\A^\mu) = u^n\Z[u^{-1}],\qquad
   FH_m^{S^1}(\A^\mu) = u^m\Z[u].
\end{gather*}

\subsection{$S^1$-equivariant Tate homology}\label{ss:Tate}
In this section we give a Morse theoretic argument for
Borel's localization theorem for circle actions,
cf.~\cite{atiyah-bott}.  
Consider a closed manifold $M$ with a circle action. Choose on $M$  a nonnegative $S^1$-invariant
Morse-Bott function $f:M\to\R$ with the following properties:
\begin{enumerate}
\item[(i)] $f(x)=0$ if and only if $x \in \mathrm{Fix}(S^1)$;
\item[(ii)] the components of the critical manifold of positive action 
consist of Bott nondegenerate
critical circles on which the action is locally free and thus defines
a finite cover $S^1\to S^1$. 
\end{enumerate}
The existence of such a function follows from the results of Wasserman
in \cite{wasserman}.

We now define the $S^1$-equivariant Tate complex of $(M,f)$ by a Morse
theoretic version of the Borel 
construction using the Rabinowitz Floer complex on $\C$. Set
$\LL_\C^*:=\LL_\C\setminus\{0\}$ and let 
$$
   \wh M := (M\times\LL_\C^*\times\R)/S^1
$$
be the quotient by the (free) diagonal circle action on
$M\times\LL_\C$. Since $f$ and $\A^\mu$ are $S^1$-invariant, they
descend to functions 
$$
   \wh f,\wh\A^\mu:\wh M\to\R
$$
that are invariant and Morse-Bott with respect to the anti-diagonal
circle action. We define the {\em $S^1$-equivariant Tate complex}
$\wh C^{S^1}(M,f)$ of $(M,f)$ as the (non-equivariant) Morse chain
complex of $\wh f+\wh\A^\mu:\wh M\to\R$. Its homology is the 
{\em $S^1$-equivariant Tate homology} $\wh H^{S^1}(M,f)$.

More precisely, we pick a generic family $g$ of metrics on the fibres
of the fibration $M\to (M\times S^\infty_\infty)/S^1\to\C P^\infty_\infty$ which is
invariant under the $\Z$-action on $(M\times S^\infty_\infty)/S^1$ given by
$n*[x,z]=[x,n*z]$. Note that such $g$ can be constructed
inductively over the strata $\C P_m^n$ of increasing dimensions.
(If $g$ is already constructed over the $\C P_m^n$ with $n-m < k$ in
a $\Z$-invariant way, we extend it arbitrarily to $\C P_0^k$ and then
by $\Z$-invariance to all $\C P_m^n$ with $n-m=k$; this is possible
because $g$ is already defined on the intersections $\C
P_m^n\cap\C P_{m+\ell}^{n+\ell}=\C P_{m+\ell}^n$ for all $\ell\geq
0$.) 
Consider the pullback diagram
\begin{equation*}
\begin{xy}
\xymatrix{
M \ar[r] & (M\times \LL_\C^*\times\R)/S^1 \ar[d] \ar[r] & (\LL_\C^*\times\R)/S^1 \ar[d] \\
M \ar[r] & (M\times S^\infty_\infty)/S^1 \ar[r] & \C P^\infty_\infty,
}
\end{xy}
\end{equation*}
where the vertical maps are induced by the normalization map
$\LL_\C^*\to S^\infty_\infty$, $z\mapsto z/\|z\|$. Thus $g$ induces a
family of metrics on fibres of the fibration $M\to\wh
M\to(\LL_\C^*\times\R)/S^1$, which combines with the pullback of the
$L^2$-metric on $\LL_\C$ and the standard metric on $\R$ 
(with respect to some choice of horizontal subspaces) 
to a metric
on $\wh M$. The boundary operator of the Morse chain complex of $\wh
f+\wh\A^\mu$ counts negative gradient flow lines with respect to such
a metric. Here are some properties of this construction. 

(0) Since $\wh f+\wh\A^\mu$ is still invariant under the anti-diagonal
$S^1$-action on $\wh M$, its critical points appear in Bott
nondegenerate families of the following types:
\begin{enumerate}
\item[(i)] $C\times\{[w^{(\ell)}]\}$ for a component $C\subset M$ of
  the fixed point set;
\item[(i)] $(\gamma\times S^1\cdot w^{(\ell)})/S^1$ for a critical
  circle $\gamma\subset M$ of $f$ outside the fixed point set.
\end{enumerate}
To define the Morse chain complex, we pick additional Morse functions
on all these critical components and count cascades of negative
gradient flow lines. 

(1) The $\mathbb{Z}$-action on the Rabinowitz Floer chain complex
induces a $\mathbb{Z}$-action on the Tate chain complex which gives
Tate homology the structure of a module over $\Lambda=\Z[u,u^{-1}]$. 

(2) Under the Tate boundary operator, both $\wh f$ and $\wh\A^\mu$ are
nonincreasing. Since gradient flow lines of $\wh f+\wh\A^\mu$ project
onto gradient flow lines of $\A^\mu$, the index of critical points of
$\A^\mu$ is also nonincreasing. However, the index of critical points
of $f$ may increase due to the fact that not every metric on $M$ in
the generic family $g$ has to be generic. 

(3) The filtration by values of $\wh f$ yields a spectral sequence.
Since $M$ is a closed manifold, the Morse-Bott function is bounded
from above and below, and hence the spectral sequence is bounded.
Therefore it converges to Tate homology, see \cite[Theorem 5.5.1]{weibel}.
Its first page is the direct sum of contributions from the critical
components of $f$, with boundary operator given by the Rabinowitz
Floer boundary operator. Let us compute the corresponding ``local
homologies''.  

If $p\in M$ is a fixed point of the $S^1$-action (and thus a critical
point of $f$), then the local homology
is just the equivariant Rabinowitz Floer homology
$\Lambda=\Z[u,u^{-1}]$, shifted by the index of $p$. 
In particular, if the circle action on $M$ is trivial, then $\wh
H^{S^1}(M,f)=H(M;\Lambda)$ is just the homology of $M$ with
coefficients in $\Lambda$.  

Consider now a critical circle $\gamma$ of $f$ on which the circle
action is an $n$-fold covering, $n\in\N$. Then $(\gamma\times
S^\infty_\infty)/S^1$ is the infinite dimensional lens space
$S^\infty_\infty/\Z_n$ obtained as the quotient by the stabilizer subgroup
$\Z_n\subset S^1$. Picking a Morse function with two critical points
on each critical fibre of the degree $n$ circle bundle
$S^\infty_\infty/\Z_n\to\C P^\infty_\infty$ gives us generators $w_\pm^\ell$ of
indices $|w_-^\ell|=2\ell$ and $|w_+^\ell|=2\ell+1$, for $\ell\in\Z$. 
To compute the
boundary operator on this complex, we need to distinguish two cases. 
Let us call $\gamma$ {\em good} if the tangent bundle to the unstable
manifold of $f$ along $\gamma$ is orientable, and {\em bad} otherwise
(the latter can only happen for $n$ even, see the following proof). 

\begin{lemma}\label{lem:circle}
If $\gamma$ is good, then the boundary maps on the first page are
given by\footnote{
We actually only determine the coefficients $\pm n,\pm 2$ up to
signs. They can arranged to be positive by replacing some generators by
their negatives if necessary. The same remark applies to subsequent
computations.}  
$$
   \cdots \stackrel{\cdot n}\longrightarrow w_+^\ell 
   \stackrel{0}\longrightarrow w_-^\ell
   \stackrel{\cdot n}\longrightarrow w_+^{\ell-1} 
   \stackrel{0}\longrightarrow w_-^{\ell-1}
   \stackrel{\cdot n}\longrightarrow \cdots 
$$
If $\gamma$ is bad, then the boundary maps on the first page are
given by
$$
   \cdots \stackrel{0}\longrightarrow w_+^\ell 
   \stackrel{\cdot 2}\longrightarrow w_-^\ell
   \stackrel{0}\longrightarrow w_+^{\ell-1} 
   \stackrel{\cdot 2}\longrightarrow w_-^{\ell-1}
   \stackrel{0}\longrightarrow \cdots 
$$
\end{lemma}

\textbf{Proof: }
Pick a point $x$ on $\gamma$ and denote by $\Phi_x:T_xM\to T_xM$ the
linearization of the $S^1$-action at time $1/n$. Since $\Phi_x$
preserves the tangent space $E_x^-$ to the unstable manifold at $x$,
the unstable bundle $E^-\to\gamma$ along $\gamma$ is isomorphic to the
bundle $[0,1]\times E_x^-/(0,v)\sim(1,\Phi_x\cdot v)$. So $E^-$ is
non-orientable (i.e., $\gamma$ is bad) if and only if
$\det(\Phi_x|_{E_x^-})=-1$. Now $\Phi_x^n=\Id$ implies $(-1)^n=1$,
which is only possible if $n$ is even. Consider the Morse complex of a
Morse function with two critical points on $\gamma$ with local
coefficients in $E^-$ (or rather its orientation bundle). The two
gradient trajectories from the maximum to the minimum occur in the
boundary operator with opposite signs if $\gamma$ is good, and with
the same sign if $\gamma$ is bad. Hence the (non-equivariant)
Morse homology of $\gamma$ with local coefficients in $E^-$ equals 
\begin{align*}
   H_*(S^1;E^-) &= \begin{cases}
      \Z & *=0 \cr \Z & *=1
   \end{cases}\qquad \text{if $\gamma$ is good}, \cr
   H_*(S^1;E^-) &= \begin{cases}
      \Z_2 & *=0 \cr \{0\} & *=1
   \end{cases}\qquad \text{if $\gamma$ is bad}.
\end{align*}
To compute the equivariant homology, let us identify $\gamma\cong S^1=\R/\Z$ with
the $S^1$-action $(\tau,t)\mapsto t+n\tau$. The associated critical
manifold $(\gamma\times S^\infty_\infty)/S^1$ is diffeomorphic to the
doubly infinite lens space $L_n:=(S^\infty_\infty)/\Z_n$ via the map
sending $[t,z]$ to $[e^{-2\pi it/n}z]$, with inverse map $[z]\mapsto
[0,z]$. Under this diffeomorphism the loop $t\mapsto[t,z]=[0,e^{-2\pi
it/n}z]$, $t\in[0,1]$ (with $z\in S^\infty_\infty$ fixed), corresponds
to a generator of $\pi_1(L_n)=\Z_n$. Thus the pullback bundle
$\pi^*E^-\to L_n$ under the projection $[t,z]\mapsto t$ is
orientable over the generator of $\Z_n$ if and only $E^-$ is orientable. 
The Morse-Bott function $h(z)=\sum_{k\in\Z}k|z_k|^2$ on $L_n$ has
critical circles $w^\ell=\{z_j=0\text{ for }j\neq \ell\}$ of index
$2\ell$ for each $\ell\in\Z$. Perturbing $h$ by Morse functions with
two critical points on each critical circle, we obtain generators
$w_+^\ell$ of index $2\ell+1$ and $w_-^\ell$ of index $2\ell$. For
each $\ell$ there are two gradient lines from $w^\ell_+$ to
$w^\ell_-$, and $n$ gradient lines from $w^{\ell+1}_-$ to $w^\ell_+$.  

For $n=1$ the lens space is the sphere and its homology is the
Rabinowitz Floer homology of $\mathcal{A}^\mu$, which vanishes
by~\cite{cieliebak-frauenfelder}. For $n\geq 2$, we again distinguish
two cases. 

If $\gamma$ is good, then the bundle $\pi^*E^-\to L_n$ is orientable
and the homology of $L_n$ with coefficients in $\pi^*E^-$ is just the
ordinary homology of the doubly infinite lens space, which by the 
computation in~\cite[Example 2.43]{hatcher} is given by  
$$
   H_*(L_n;\pi^*E^-) = \begin{cases}
      \{0\} & * \text{ even} \cr \Z_n & * \text{ odd}
   \end{cases}\qquad \text{if $\gamma$ is good}. 
$$
If $\gamma$ is bad, then the bundle $\pi^*E^-\to L_n$ is
non-orientable over each circle $w^\ell$ (which represents a generator
of $\pi_1(L_n)$), so the Morse boundary operator $\p$ maps $w^\ell_+$
to $2w^\ell_-$. The relation $\p\circ\p=0$ then enforces $\p
w^\ell_-=0$, so the homology is given by  
$$
   H_*(L_n;\pi^*E^-) = \begin{cases}
      \Z_2 & * \text{ even} \cr \{0\} & * \text{ odd}
   \end{cases}\qquad \text{if $\gamma$ is bad}.
$$
Since the chain complex has only one generator in each degree, in
order to yield these homology groups the boundary operator must have
(up to signs) the multiplicities given in the lemma. 
\hfill$\square$

{\em Remark. }
The homology of the lens space $L_n$ with local coefficients in
$\pi^*E^-$ in the preceding proof can also be obtained more
algebraically as follows. The vanishing of Rabinowitz Floer 
homology of $\mathcal{A}^\mu$ implies that the local chain complex is
a complete resolution of the group $\mathbb{Z}_n$ in the sense of 
\cite{brown}. If the periodic orbit $\gamma$ is good, then with
integer coefficients the local Tate homology is given by the Tate
homology of the group $\mathbb{Z}_n$,
$$\widehat{H}_*(\mathbb{Z}_n;\mathbb{Z})=\left\{\begin{array}{cc}
\{0\} & *\,\,\mathrm{even}\\
\mathbb{Z}_n & *\,\,\mathrm{odd}.
\end{array}\right.
$$
If $\gamma$ is bad, then $n$ is necessarily even and we have to consider the
twisted $\mathbb{Z}\mathbb{Z}_n$-module $\widehat{\mathbb{Z}}$ where the generator of
$\mathbb{Z}_n$ acts via $-1$. In this case the local Tate homology is
given by 
$$\widehat{H}_*(\mathbb{Z}_n;\widehat{\mathbb{Z}})=\left\{\begin{array}{cc}
\mathbb{Z}_2 & *\,\,\mathrm{even}\\
\{0\} & *\,\,\mathrm{odd}.
\end{array}\right.
$$

\begin{cor}\label{cor:circle}
The contribution of a critical circle $\gamma$
of covering number $n$ to the homology of the first page equals 
\begin{itemize}
\item $\{0\}$ if $n=1$, 
\item $\Z_n[u,u^{-1}]$ shifted by $|\gamma|+1$ if $n\geq 2$ and
$\gamma$ is good, 
\item $\Z_2[u,u^{-1}]$ shifted by $|\gamma|$ if $\gamma$ is bad. 
\end{itemize}
In particular, if the circle action on $M$ is free,
then $\wh H^{S^1}(M,f)=0$.
\hfill$\square$ 
\end{cor}

With $\Q$-coefficients, the contribution of each critical circle $\gamma$
to the homology of the first page vanishes, so only the contributions
from the fixed points remain. The boundary operator on the second page
counts gradient flow lines of $f$ connecting Morse critical points. 
After this, the spectral sequence collapses. Indeed, since the Morse-Bott function
$f$ is constant $0$ on the fixed points, different components of the fixed point
set cannot interact via gradient flow lines. We have thus derived Borel's localization theorem
(see~\cite{atiyah-bott}):

\begin{cor}
The Tate homology with $\Q$-coefficients equals the homology of the
fixed point set $M_{S^1}$ with coefficients in $\Lambda_\Q:=\Q[u,u^{-1}]$,
$$
   \wh H^{S^1}(M,f;\Q) \cong H(M_{S^1};\Lambda_\Q) \cong
   H(M_{S^1})\otimes_\Q\Q[u,u^{-1}]. 
$$ 
In particular, if the circle action on $M$ has no fixed points, then
$\wh H^{S^1}(M,f;\Q)=0$. 
\end{cor}

(4) Using the filtration by values of the Rabinowitz action functional
$\wh\A^\mu$ (not of $\wh f+\wh\A^\mu$!), we define filtered Tate
homology groups $\wh H_a^b(M,f)$, $\wh H_a(M,f)$ and $\wh H^b(M,f)$
with the obvious notation as above. Note that for integers $m\leq n$,
multiplication by $u^m$ defines canonical isomorphisms (of degree $2m$)
$$
   \wh H_0^{n-m}(M,f)\cong \wh H_m^n(M,f),\qquad
   \wh H_0(M,f)\cong \wh H_m(M,f).
$$
By definition, $\wh H_0^n(M,f)$ equals the Morse homology of the
manifold $(M\times S^{2n+1})/S^1$. Since the homology functor and the
direct limit functor commute, it follows that 
$$
   \wh H_0(M,f) 
   = \lim_{\substack{\longrightarrow\\ n \to \infty}}\wh H_0^n(M,f) 
   = \lim_{\substack{\longrightarrow\\ n \to \infty}}H\Bigl(M\times
   S^{2n+1})/S^1\Bigr)  
   = H^{S^1}(M)
$$
equals the $S^1$-equivariant homology of $M$. Here
$H^{S^1}(M)=H(M_{S^1})$ is defined via the Borel construction, see
below. So we have shown

\begin{prop}\label{prop:tate-symp}
For every $a\in\R$, the filtered Tate homology $\wh H_a(M,f)$ is
canonically isomorphic to the $S^1$-equivariant homology $H^{S^1}(M)$
with degrees shifted by $2[a]$. \hfill$\square$
\end{prop}

Recall the Borel construction $X_G=X\times_GEG$ for a $G$-space $X$. It fits into the diagram
\begin{equation}
\begin{xy}
\xymatrix{
X\times EG
\ar[d]^{\pi}
\ar[r]
& EG
\ar[d]
\\
X_G   \ar[r]^p &
BG,
}
\end{xy}
\end{equation}
where the vertical maps are principal $G$-bundles and the horizontal
maps are induced by projection onto the second factor. Pullback under
$p$ yields a ring homomorphism $p^*:H^*(BG)\to H^*(X_G)$, which makes
equivariant cohomology a module over $H^*(BG)$ via cup product. The
image $p^*e$ of a class $e\in H^*(BG)$ is the characteristic class of
the bundle $\pi$ induced by $e$, and cap product with these classes
makes equivariant homology $H_*(X_G)$ a module over $H^*(BG)$ as well.  

For $G=S^1$ and $X$ being a manifold, the class $p^*e$ corresponding
to the generator $e\in H^2(BS^1)$ is the Euler class of $\pi$, and cap
product with this class is realized by intersection with the
codimension $2$ submanifold of $X_{S^1}$ given by a complex
codimension $1$ linear subspace of $BS^1=\C P_0^\infty$.   

\begin{cor}\label{cor:tate-symp}
The isomorphism in Proposition~\ref{prop:tate-symp} is compatible with
the module structures over the polynomial ring $\Z[u^{-1}]$, where
$u^{-1}$ acts on $H^{S^1}(M)$ by cap product with the Euler class
$e\in H^2(BS^1)$.
\end{cor}

{\bf Proof: }
Define a complex codimension $1$ linear subspace $H$ of $\C P_0^\infty$
by the equation $\sum_{k=0}^\infty a_kz_k=0$ on Fourier coefficients,
for an $\ell^2$ sequence $(a_k)$ with $a_k\neq 0$ for all
$k$. Gradient flow lines of $\A^\mu$ connecting two critical points
$(z^{(\ell+1)},\ell+1)$ and $(z^{(\ell)},\ell)$  project onto the
$2$-sphere $\{z_k=0$ for all $k\neq\ell,\ell+1\}$ in $\C P_0^\infty$,
which intersects $H$ in the single point $\{z_k=0$ for all
$k\neq\ell,\ell+1,\;a_\ell z_\ell+a_{\ell+1}z_{\ell+1}=0\}$. Thus cap
product with the Euler class maps $(z^{(\ell+1)},\ell+1)$ to
$(z^{(\ell)},\ell)$, which is also what the map $u^{-1}$ does.   
\hfill$\square$

\section{Symplectic Tate homology}

In this section we carry over the construction of
Section~\ref{ss:Tate} to define symplectic Tate homology, and we prove
Theorem~\ref{thm:main} from the Introduction. 

\subsection{Definition and basic properties}

Let $(V,\lambda_V)$ be the completion of a Liouville domain
$(W,\lambda)$ satisfying $c_1(W)=0$. Let $H \in
C^\infty(V,\mathbb{R})$ be an autonomous Hamiltonian growing
quadratically at infinity. Denote by $\mathscr{L}_V :=
C^\infty(S^1,V)$ 
the free loop
space of $V$. Since the Hamiltonian is autonomous, the action
functional of classical mechanics 
$$\mathcal{A}_H \colon \mathscr{L}_V \to \mathbb{R}$$
defined by
$$\mathcal{A}_H(v):=\int v^* \lambda_V -\int H(v)dt$$
is invariant under the circle action of $\mathscr{L}_V$ induced by
rotating the domain. We assume in the following that it is Morse--Bott,
which we can always achieve by perturbing $H$ slightly. 

Recall from Section~\ref{ss:Rab} the definition of the Rabinowitz
action functional for the unit circle in $\mathbb{C}$,  
$$ 
   \mathcal{A}^\mu \colon \mathscr{L}_{\mathbb{C}} \times \mathbb{R}
   \to \mathbb{R}, \qquad
   \mathcal{A}^\mu(z,\eta)=\int z^* \lambda_{\mathbb{C}}-\eta\int
   \mu(z)dt.
$$
The circle acts on the product space $\mathscr{L}_{V} \times
\mathscr{L}_{\mathbb{C}}\times \mathbb{R}$ diagonally.
Since both $\mathcal{A}_H$ and $\mathcal{A}^\mu$ are $S^1$-invariant,
the functional 
$$\mathcal{T}_H \colon (\mathscr{L}_{V} \times \mathscr{L}_{\mathbb{C}}\times
   \mathbb{R})/S^1 \to \mathbb{R}$$
given by
$$
\mathcal{T}_H([v,z,\eta])=\mathcal{A}_H(v)+\mathcal{A}^\mu(z,\eta)$$
is well defined. Since $\mathcal{A}_H$ was assumed to be Morse-Bott, 
$\mathcal{T}_H$ is Morse-Bott as well. Choose an auxiliary Morse
function $f$ on the critical manifold of  
$\mathcal{T}_H$. On the set $\mathrm{Crit}(f)$ of critical points of
the auxiliary Morse function $f$ we define a double filtration by
$$
   \mathrm{Crit}_a^b(f):=\Big\{c \in \mathrm{Crit}(f) \mid
   \mathcal{A}^\mu(c) \geq a, \,\, \mathcal{A}_H(c) \leq b\Big\},
   \quad a,b \in \mathbb{R}. 
$$
Since $c_1(W)=0$, elements in $\mathrm{Crit}(f)$ are naturally graded
by the sum of the Conley--Zehnder indices of the periodic orbits for
$\A_H$ and $\A^\mu$ and the Morse index of the critical points of
$f$. (Here the Conley-Zehnder indices of periodic orbits of $\A_H$
depend on a choice of trivializations of $TV$ along base loops in each
free homotopy class.) 
It is important to observe that for all
$a,b \in \mathbb{R}$ the set $\mathrm{Crit}_a^b(f)$ is finite in each fixed degree. We omit the 
reference to the degree in our notation. The {\em symplectic Tate
  chain groups} (with $\Z$-coefficients) are
defined as the free $\Z$-modules generated by
critical points of our auxiliary Morse function, 
$$
   CT_a^b 
   := \mathrm{Crit}_a^b(f) \otimes \Z.
$$
The boundary operator 
$$\partial_a^b \colon CT_a^b \to CT_a^b$$
is defined by counting gradient flow lines with cascades for 
the pair $(\mathcal{T}_H,f)$ as in \cite{bourgeois-oancea1} or \cite[Appendix A]{frauenfelder0}.
This definition depends on the choice of a suitable family of almost
complex structures, which we defer to Section~\ref{ss:trans}. 
We denote its homology groups by
$$HT_a^b := \frac{\ker \partial_a^b}{\mathrm{im} \partial_a^b}.$$
Since both action functionals $\mathcal{A}_H$ and $\mathcal{A}^\mu$
are nonincreasing along gradient flow lines with cascades, the
inclusion and projection maps 
$$
   \iota_a^{b_2,b_1} \colon CT_a^{b_1} \to CT_a^{b_2}, \,\,b_1 \leq
   b_2, \qquad
   \pi_{a_2,a_1}^b \colon CT^b_{a_1} \to CT_{a_2}^b,\,\,a_1 \leq a_2
$$
together with the boundary operators $\p_a^b$ form a bidirect system
of chain complexes as in Section~\ref{sec:alg}. In particular, the
induced maps on homology
$$
   H\iota_a^{b_2,b_1} \colon HT_a^{b_1} \to HT_a^{b_2}, \,\,b_1 \leq
   b_2, \qquad 
   H\pi_{a_2,a_1}^b \colon HT^b_{a_1} \to HT_{a_2}^b,\,\,a_1 \leq a_2
$$
form a bidirect system of graded vector spaces. 
As in Section~\ref{ss:bidirect}, we define four versions of {\em
  symplectic Tate homology} by 
\begin{align*}
\overrightarrow{H}\overleftarrow{T}(W) &:=
H(\underrightarrow{\lim}\underleftarrow{\lim}CT), \qquad
\underrightarrow{H}\underleftarrow{T}(W) :=  \underrightarrow{\lim}
\underleftarrow{\lim}HT\quad \text{(Jones-Petrack version)}, \cr
\overleftarrow{H}\overrightarrow{T}(W) &:= H(\underleftarrow{\lim}\underrightarrow{\lim}CT) , \qquad
\underleftarrow{H}\underrightarrow{T}(W) :=
\underleftarrow{\lim}\underrightarrow{\lim}HT\quad \text{(Goodwillie
  version)}. 
\end{align*}
Thus superscript (resp.~subscript) arrows indicate that the direct/inverse
limits are applied before (resp.~after) taking homology. As the
notation suggests, the symplectic Tate homology groups
are independent of $H$, the auxiliary Morse function $f$, 
as well as the metrics used in order to define gradient flow lines
with cascades. This can be shown by standard continuation arguments in
Morse homology as in \cite{schwarz}. By construction, they are
invariant under Liouville isomorphisms of the completions, hence in
particular under deformation equivalence of the Liouville domains (see
also~\cite{Sei08b}). 

Note that the spaces $C^b_a$ are finitely generated in each degree and
the projections $\pi_{a_2,a_1}^b$ are surjective. Hence
Proposition~\ref{ta} and Corollary~\ref{cor:module} imply  

\begin{thm}\label{thm:maindia}
Let $W$ be a Liouville domain with $c_1(W)=0$. Then there is a 
canonical diagram  
\begin{equation}\label{maindia2}
\begin{xy}
 \xymatrix{
\overrightarrow{H}\overleftarrow{T}(W)
\ar[d]^{Hk}
\ar[r]^{\rho}
& \underrightarrow{H}\underleftarrow{T}(W)
\ar[d]^{\kappa}\\
\overleftarrow{H}\overrightarrow{T}(W)   \ar[r]^{\sigma} &
\underleftarrow{H}\underrightarrow{T}(W)
}
\end{xy}
\end{equation}
of $\Z[u,u^{-1}]$-module maps with $\sigma$ surjective. With
coefficients in a field, the diagram commutes and $\rho$ is an
isomorphism. \hfill$\square$ 
\end{thm}

This proves parts (a) and (b) of Theorem~\ref{thm:main}. 

\subsection{Localization properties}

Next we turn to the proof of Theorem~\ref{thm:main} (c) and (d). 
For $a \in \mathbb{R}$ we abbreviate
$$HT_a(W)=HT_a:=\lim_{\substack{\longrightarrow\\ b}}{HT}_a^b.$$
Again, $HT_a$ only depends on $W$ up to Liouville isomorphisms of the
completions. 
The next proposition identifies these groups with the
$S^1$-equivariant symplectic homology $SH^{S^1}(W)$ of
the Liouville domain $W$ as 
defined by Bourgeois and Oancea in \cite{bourgeois-oancea}.

\begin{prop}\label{symp}
For each $a\in \mathbb{R}$ 
we have a canonical isomorphism of $\Z[u^{-1}]$-modules
$$
   HT_a(W)\cong SH^{S^1}(W).
$$
\end{prop}

\textbf{Proof: }Because of the $\mathbb{Z}$-invariance of the chain
complex it is clear that the $HT_a$ for different choices of $a$ are
canonically isomorphic. It therefore suffices to show that $HT_0$ is
canonically isomorphic to $SH^{S^1}$. We filter the action of
$\mathcal{A}^\mu$ also from above and denote for $n \in \mathbb{N}$
the corresponding vector space by $H^nT_0$. Since the homology functor
and the direct limit functor commute, we have 
$$
   \lim_{\substack{\longrightarrow\\ n \to \infty}}H^nT_0=HT_0.
$$
For a loop $z \in \mathscr{L}_\mathbb{C}$ we consider its Fourier expansion
$$z(t)=\sum_{k=-\infty}^\infty z_k e^{2\pi i k t}, \quad z_k \in \mathbb{C}.$$
Consider the finite dimensional subspace $(\mathscr{L}_\mathbb{C})_0^n
\cong \mathbb{C}^{n+1}$ 
of the loop space $\mathscr{L}_\mathbb{C}$ consisting of all Fourier
series with nonvanishing Fourier coefficients only in the range from
$0$ to $n$. 
The gradient flow equation~\eqref{eq:Fourier} for $\A^\mu$ in terms of
the Fourier coefficients shows that the subspace
$(\mathscr{L}_\mathbb{C})^n_0$ is invariant under the gradient flow of
$\mathcal{A}^\mu$. 
The restriction of $\mathcal{A}^\mu$ to the subspace $(\mathscr{L}_\mathbb{C})^n_0$ is given by
the Lagrange multiplier function
$$\mathcal{A}^\mu|_{(\mathscr{L}_\mathbb{C})^n_0}(z,\eta)=\bar\A(z)+\eta\bar{\mu}(z)$$
where
$$
   \bar\A(z)=\pi\sum_{k=0}^n k|z_k|^2, \qquad
   \bar{\mu}(z)=\pi\Bigl(\sum_{k=0}^n |z_k|^2-1\Bigr).
$$
Note that $\bar{\mu}^{-1}(0)=S^{2n-1}$. Now homotope the Morse homology
of the Lagrange multiplier functional to the Morse homology of the
function $\bar\A$ on the constraint $\bar{\mu}^{-1}(0)$ as in  
\cite{frauenfelder} to recover the definition of $S^1$-equivariant
symplectic homology of Bourgeois and Oancea \cite{bourgeois-oancea,
  bourgeois-oancea2}.
The resulting isomorphism is compatible with the 
$\Z[u^{-1}]$-module structures, where the action of $u^{-1}$ on
$SH^{S^1}(W)$ is induced by the cap product with the Euler class $e\in
H^2(BS^1)$ in the definition by Bourgeois and Oancea, which is modeled
on the Borel construction.  
\hfill $\square$


Now we specialize to coefficients in a field $\fk$. The following
corollary establishes Theorem~\ref{thm:main} (d). 

\begin{cor}\label{local}
For any field $\mathfrak{k}$, the Goodwillie version of Tate homology
$\underleftarrow{H}\underrightarrow{T}(W;\mathfrak{k})$ coincides with
the localization of $S^1$-equivariant symplectic homology of $W$. 
\end{cor}

\textbf{Proof: }
Recall that the map $u$ defines isomorphisms $u:CT_a^b\to
CT_{a+\pi}^{b+\pi}$ on the filtered Rabinowitz Floer complex. Consider 
the vector space $V:=CT_0^\infty$ and the linear map
$$
   T:=p_{-1}\circ u^{-1}|_V:V\to V,
$$ 
where $p_{-1}:=\pi_{0,-\pi}^\infty:CT_{-\pi}^\infty\to CT_0^\infty$ is the
canonical projection. Since $T$ commutes with the boundary operator
$\partial:=\partial^\infty_0$, the triple  
$(V,T,\partial)$ is a Tate triple as in Definition~\ref{tatetriple}. 
Since $T:V\to V$ is surjective, Corollary~\ref{locmain} yields the
isomorphism of $\mathfrak{k}[u,u^{-1}]$-modules
\begin{equation}\label{eq:loc}
   \lim_{\substack{\longleftarrow\\ k \to \infty}}H(V_k^T,\partial_k)\cong H(V,\partial)^{HT}.
\end{equation}
In view of Proposition~\ref{symp} and $H(V,\partial)=HT_0(W;\mathfrak{k})$, the
right hand side of~\eqref{eq:loc} is the localization of $S^1$-equivariant symplectic
homology of $W$. To understand the left hand side, recall from
Section~\ref{loc} that $V_k^T$ is the quotient of the space of
sequences $(v_i)_{i\in\N}$ of $v_i\in V$ with $Tv_{i+1}=v_i$ by those
sequences with $v_k=0$. We claim that we have a commuting diagram
\begin{equation*}
\begin{xy}
\xymatrix{
V_k^T \ar[d]^{f_k}_\cong \ar[r]^{\pi_k} 
& V_{k-1}^T \ar[d]^{f_{k-1}}_\cong \\
CT_{-k\pi}^\infty \ar[r]^{p_{-k}} 
& CT_{-(k-1)\pi}^\infty. \\
}
\end{xy}
\end{equation*}
Here the horizontal maps are the canonical projections and the
vertical map $f_k$ sends $[(v_i)]\in V_k^T$ to 
$u^{-k}v_k\in CT_{-k\pi}^\infty$. For injectivity of $f_k$, note that
$0=u^{-k}v_k\in CT_{-k\pi}^\infty$ implies $0=v_k\in CT_0^\infty$ and
thus $0=[(v_i)]\in V_k^T$. For surjectivity, a preimage of $x\in
CT_{-k\pi}^\infty$ is obtained by setting $v_k:=u^kx\in CT_0^\infty$ and
extending it to $[(v_i)]\in V_k^T$ using surjectivity of $T$. 
Commutativity of the diagram follows from 
\begin{align*}
   f_{k-1}\pi_k[(v_i)] &= u^{-(k-1)}v_{k-1} = u^{-(k-1)}Tv_k =
   u^{-(k-1)}p_{-1}u^{-1}v_k \cr
   &= p_{-k}u^{-(k-1)}u^{-1}v_k = 
   p_{-k}f_k[(v_i)].  
\end{align*}
In view of the commuting diagram, the left hand side of~\eqref{eq:loc}
becomes
$$
   \lim_{\substack{\longleftarrow\\ k \to \infty}}H(V_k^T,\partial_k)
   \cong \lim_{\substack{\longleftarrow\\ a \to
       -\infty}}H(CT_a^\infty,\partial_a^\infty) 
   \cong \lim_{\substack{\longleftarrow\\ a \to
       -\infty}}\lim_{\substack{\longrightarrow\\ b \to \infty}}HT_a^b 
   \cong \underleftarrow{H}\underrightarrow{T}(W;\mathfrak{k}), 
$$
where the second isomorphism comes from the fact that the direct limit
commutes with the homology functor. This proves the corollary. 
\hfill $\square$
\medskip

Finally, we prove that with rational coefficients
$\overrightarrow{H}\overleftarrow{T}$ has the fixed point property
\cite{cencelj, greenlees-may, jones-petrack}, which in view of the
isomorphism $\overrightarrow{H}\overleftarrow{T}\cong
\underrightarrow{H}\underleftarrow{T}$ in Theorem~\ref{thm:maindia}
establishes Theorem~\ref{thm:main} (c). 

\begin{prop}\label{fix}
For any Liouville domain $W$ satisfying $c_1(W)=0$, the Jones-Petrack
version of symplectic Tate homology has the fixed point property
$$
   \overrightarrow{H}\overleftarrow{T}(W;\mathbb{Q}) \cong
   H(W,\partial W;\mathbb{Q}) \otimes_\Q \mathbb{Q}[u,u^{-1}].
$$
\end{prop}

\textbf{Proof: }
We consider the spectral sequence associated to the increasing filtration 
$\FF^b:=\{\mathcal{A}_H\leq b\}$ on the chain complex
$\underrightarrow{\lim}\underleftarrow{\lim}CT$, see \cite[Chapter
  5.4]{weibel}. Since $H$ grows quadratically at infinity, the
filtration is bounded from below. Moreover, it is exhaustive in the
sense that $\bigcup_{b\geq
  0}\FF^b=\underrightarrow{\lim}\underleftarrow{\lim}CT$. (Note that
this fails for the chain complex
$\underleftarrow{\lim}\underrightarrow{\lim}CT$.) Therefore, the 
spectral sequence converges to $\overrightarrow{H}\overleftarrow{T}(W;\mathbb{Q})$. 

The spectral sequence depends on the choice of our Hamiltonian. We
explain next how to choose the Hamiltonian $H$ on the completion $V$
of the Liouville domain $W$ so that the result can easily be read off
from the spectral sequence.  
After a small perturbation of $W$ we may assume without loss of generality that the Reeb flow
on $\partial W$ is nondegenerate. Now choose a smooth function $\beta
\in C^\infty([0, \infty), [0,\infty))$ satisfying 
$$\beta(r)\left\{\begin{array}{cc}
=0 & r \leq 1\\
>0 & r>1\\
=r & r \geq 2.
\end{array}\right.$$
Define $h \in C^\infty([0, \infty), [0,\infty))$ by the requirement
$$h(0)=0, \qquad h'(0)=0, \qquad h''(r)=\frac{\beta(r)}{r}\text{
      for }r>0.$$
Note that $h$ grows quadratically at infinity. Moreover, since $h$ is
strictly convex for $r >1$, it satisfies 
\begin{equation}\label{actest}
h'(r) r-h(r) > 0\qquad\text{ for all }r>1.
\end{equation}
If $(\partial W \times (0,\infty), d(r \lambda|_{\partial W}))$ is embedded as the symplectization of the Liouville domain $W$ in its completion $V$ we define the Hamiltonian $H \in C^\infty(V,\mathbb{R})$ as
$$H(v)=\left\{\begin{array}{cc}
h(r) & v=(x,r) \in \partial W \times (0,\infty),\\
0 & v \in W.
\end{array}\right.$$
The critical set of the action functional $\mathcal{A}_H$ associated to $H$ can be identified with $W$ corresponding to the constant loops and a disjoint union of circles, one for each periodic Reeb orbit
of the Reeb flow on $\partial W$. In particular, the action functional $\mathcal{A}_H$ is not quite Morse-Bott in the orthodox sense, since the critical manifold $W$ has a boundary. However, note
that the action of $\mathcal{A}_H$ vanishes on $W$ and is positive on
the nonconstant periodic Reeb orbits
in view of (\ref{actest}). In view of this fact, we can define Morse-Bott homology of $\mathcal{A}_H$ as for an honest Morse-Bott functional treated in \cite{frauenfelder-schlenk}.  Namely, choose a Morse function $f$ on
$\mathrm{Crit}(\mathcal{A}_H)$ with the property that $\nabla f$ points outward at the
boundary of $W$ for one (and hence every) Riemannian metric on
$\mathrm{Crit}(\mathcal{A}_H)$. Since the action is decreasing along
gradient flow lines, no flow line with cascades can contain a Floer
cylinder starting or ending on a constant in the boundary of $W$, so
the boundary operator squares to zero.  

The Liouville domain $W$ corresponds precisely to the fixed point set of the $S^1$-action on $\mathrm{Crit}(\mathcal{A}_H)$. We can assume without loss of generality that $W$ is connected (otherwise we look
at the symplectic Tate homologies of each connected component of
$W$). Therefore, its local Tate homology (the contribution to the
first page of the spectral sequence) with rational coefficients is 
$H(W,\partial W;\mathbb{Q})  \otimes_\Q \mathbb{Q}[u,u^{-1}]$. The reason we have to
take homology relative to the boundary of $W$ is that the auxiliary Morse function points outward at the boundary.

On the other hand, by Lemma~\ref{lem:circle2} below, the contribution of
each nonconstant periodic orbit $\gamma$ of the Hamiltonian vector
field $X_H$ to the first page of the spectral sequence is torsion and
thus vanishes with $\Q$-coefficients. 
This finishes the proof of Proposition~\ref{fix}. \hfill $\square$

In the proof we have used the following analogue of
Corollary~\ref{cor:circle}, where as in~\cite{EGH} a Hamiltonian orbit
$\gamma$ is called {\em bad} if it is an even multiple of a simple orbit
whose linearized return map has an odd number of eigenvalues in the
interval $(-1,0)$, and {\em good} otherwise. 

\begin{lemma}\label{lem:circle2}
The local symplectic Tate homology (i.e., the contribution to the homology of the
first page of the spectral sequence above) of a nonconstant periodic
orbit $\gamma$ of the Hamiltonian vector field $X_H$ of covering
number $n$ equals (with $\Z$-coefficients ) 
\begin{itemize}
\item $\{0\}$ if $n=1$, 
\item $\Z_n[u,u^{-1}]$ shifted by $|\gamma|+1$ if $n\geq 2$ and
$\gamma$ is good, 
\item $\Z_2[u,u^{-1}]$ shifted by $|\gamma|$ if $\gamma$ is bad. 
\end{itemize}
\end{lemma}

{\bf Proof: }
The local symplectic Tate homology of $\gamma$ is computed by a chain 
complex with the same generators $w^\ell_\pm$, $\ell\in\Z$, and the
same gradient flow lines between them as in the proof of
Lemma~\ref{lem:circle}. However, the signs with which the gradient
flow lines contribute to the boundary operator are now determined by
the coherent orientations from~\cite{FloHof93}. Rather than analyzing
the coherent orientations, we can deduce the signs from
Lemma~\ref{lem:circle} and the following two facts:

(1) The local (non-equivariant) symplectic homology $SH(\gamma;\Q)$
  of $\gamma$ with rational coefficients vanishes if and only if
  $\gamma$ is bad; this follows from~\cite[Lemma
    4.28]{bourgeois-oancea1}.   

(2) The local non-equivariant symplectic homology $SH(\gamma;\Q)$
  of $\gamma$ with rational coefficients vanishes if and only if the
  local equivariant symplectic homology $SH^{S^1}(\gamma;\Q)$
  vanishes; this follows from the Gysin exact sequence
  in~\cite{bourgeois-oancea} and the fact that both groups live in
  nonnegative degrees. 

Now the local symplectic Tate homology of $\gamma$ is the homology of
the lens space $L_n=(S^\infty_\infty)/\Z_n$ with some local
system. For $n$ odd (so $\gamma$ is good) there is no nontrivial group
homomorphism $\Z_n\to\Z_2$, so the local system is trivial and the
homology is of the first (for $n=1$) or second type in the statement
of Lemma~\ref{lem:circle2}. Suppose now that $n$ is even, so there exists a unique
nontrivial group homomorphism $\Z_n\to\Z_2$ and thus a unique
nontrivial local system on $L_n$. 
By facts (1) and (2), $\gamma$ is bad if and only if the local
equivariant symplectic homology $SH^{S^1}(\gamma;\Q)$ with rational
coefficients vanishes. This is the case if and only if the complex
agrees with the second complex in Lemma~\ref{lem:circle} (set to zero
in negative degrees, so the first complex has rational homology in
degree $0$), so the local system on $L_n$ is nontrivial and the local
symplectic Tate homology is of the third type in the statement of
Lemma~\ref{lem:circle2}. If $\gamma$ is good, then the the local
system on $L_n$ is trivial and the local symplectic Tate homology is
of the second type in the statement of Lemma~\ref{lem:circle2}.  
\hfill $\square$

\begin{rem}
The construction of symplectic Tate homology shows: Whenever the map
$\kappa$ has nontrivial kernel, then there exist infinite chains of
periodic orbits connected by Floer cylinders with $\A_H$ going to
infinity. We will see examples of this phenomenon in
Section~\ref{ex}. 
\end{rem}

\subsection{Backwards symplectic homology}

For $a<0$, let us denote by $CB_a^b\subset CT_a^b$ the subcomplex
defined by $\A^\mu\leq 0$. We call $HB(W) :=
\overrightarrow{H}\overleftarrow{B}(W)$ the {\em 
  backwards $S^1$-equivariant symplectic homology} of $W$. We get a
short exact sequence of 
bidirect systems of chain complexes $0\to CB_a^b\to CT_a^b\to CT_0^b\to
0$ (defined for $a<0$ and with one filtration trivial on
$CT_0^b$). According to Remark~\ref{rem:short-exact}, with
coefficients in a field $\fk$ we obtain long
exact sequences fitting in the following commuting diagram with the
map $\rho$ from~\eqref{maindia2}: 
\begin{equation}\label{eq:backwards}
\begin{xy}
\xymatrix{
\cdots \ar[r] &
\overrightarrow{H}\overleftarrow{B}(W;\fk)
\ar[d]^{\rho}_\cong \ar[r] & 
\overrightarrow{H}\overleftarrow{T}(W;\fk)
\ar[d]^{\rho}_\cong \ar[r] & 
SH^{S^1}(W;\fk)
\ar[d]^{\rho}_\cong \ar[r] & \cdots
\\
\cdots \ar[r] &
\underrightarrow{H}\underleftarrow{B}(W;\fk)
\ar[r] &
\underrightarrow{H}\underleftarrow{T}(W;\fk)
\ar[r] &
SH^{S^1}(W;\fk)
\ar[r] & \cdots
}
\end{xy}
\end{equation}
The term ``backwards homology'' is taken from~\cite{tene}, where an
exact sequence analogous to~\eqref{eq:backwards} is derived in a
different context.    

\subsection{Notes on transversality}\label{ss:trans}

The definition of Tate homology of a smooth manifold (resp.~symplectic
Tate homology of a Liouville domain) $M$ depends on the choice of a
generic family of metrics (resp.~compatible almost complex structures)
on the fibres of the fibration $M\to (M\times S^\infty_\infty)/S^1\to\C
P^\infty_\infty$ which is invariant under the $\Z$-action. Here we explain
the construction of such a family of almost complex structures, the
case of a family of metrics being similar but simpler. 

Recall from Section~\ref{ss:Rab} the Hilbert space $\LL_\C$ of Fourier
series, equipped with the $W^{1,2}$-norm and the $\Z$-action
$(n*z)_k=z_{k-n}$. 

\begin{lemma}
The $\Z$-action on $\LL_\C^*=\LL_\C\setminus\{0\}$ is free and proper,
so the quotient by this action is a smooth Hilbert manifold.  
\end{lemma}

\begin{proof}
Freeness holds because an element with $n*z=z$ for some $n\neq 0$ can
only be in $W^{1,2}$ if it is zero. Recall that the $\Z$-action is
proper if and only if for all $x,y\in\LL_\C^*$ there exist open
neighbourhoods $U_x,U_y$ such that $(n*U_x)\cap U_y=\emptyset$ for all
but finitely many $n\in\Z$. To show this, consider
$x,y\in\LL_\C^*$. Pick $\ell\in\Z$ with $y_\ell\neq 0$ and let
$U_x,U_y$ be the open $\eps$-balls with respect to the $L^2$-norm,
where $\eps:=|y_\ell|/3$. Suppose that $z\in U_x$ and $n*z\in
U_y$. Then $|z_{\ell-n}-y_\ell|\leq \|n*z-y\|<\eps$, and thus
$|z_{\ell-n}|\geq 2\eps$. On the other hand, there exists $K\in\N$
such that $|x_k|<\eps$ for all $k\in\Z$ with $|k|>K$. It follows that
$|z_k|<2\eps$ for all $k\in\Z$ with $|k|>K$, which together with the
previous estimate yields $|\ell-n|\leq K$, which is only satisfied for
finitely many $n\in\Z$. The proof in~\cite{abraham-marsden} that
the quotient of a manifold by a proper free action is again a
manifold carries over readily to Hilbert manifolds.
\end{proof}

Since the free $S^1$-action on $\LL_\C^*$ commutes with the
$\Z$-action, we obtain a principal circle bundle
$$
   S^1\to \LL_\C^*/\Z \to \LL_\C^*/(S^1\times\Z). 
$$
Given a completed Liouville domain $(V,\lambda_V)$, let
$\JJ(V,\lambda_V)$ be the space of almost complex structures on $V$
compatible with $d\lambda_V$ that are cylindrical over the collar
$[1,\infty)\times\p V$. Let 
$$
   \JJ := C^\infty\Bigl(\LL_\C^*/\Z,\JJ(V,\lambda_V)\Bigr)
$$
be the space of almost complex structures in $\JJ(V,\lambda_V)$
parametrized by $\LL_\C^*/\Z$. An element $J\in\JJ$ defines by
integration of $d\lambda(\cdot,J\cdot)$ an
$L^2$-inner product on the vertical tangent bundle of the fibre bundle
$$
   \LL_V \to (\LL_V\times\LL_\C^*)/(S^1\times\Z) \to
   \LL_\C^*/(S^!\times\Z). 
$$
Identifying $\LL_\C^*/\Z$ with the total space of the associated
bundle $S^1\to (\LL_\C^*\times S^1)/(S^1\times\Z)\to
\LL_\C^*/(S^1\times\Z)$, we can think of $J\in\JJ$ as a family of
almost complex structures $J_{z,t}$ parametrized by
$(z,t)\in\LL_\C^*\times S^1$, invariant under the diagonal
$S^1$-action on $(z,t)$ and the $\Z$-action on $z$. Then the
$V$-component of a gradient flow line $[v,z,\eta]$ of the functional
$\mathcal{T}_H$ satisfies the Floer equation
$$
   \p_sv+J_{z(s),t}(v)\bigl(\p_tv-X_H(v)\bigr) = 0.
$$
This is the Floer equation on a cylinder with an $S^1$-dependent
almost complex structure. It follows by standard methods in Floer
theory (see e.g.~\cite{salamon}) that one an achieve transversality for all
moduli spaces by generic choice of $J$ in this class. Compactness
modulo breaking follows by combining compactness modulo breaking for
the components $(z,\eta)$, which satisfy the gradient flow equation
for $\A^\mu$, with Floer compactness for the component $v$ satisfying
the Floer equation above. 

Let us point out that the preeding discussion is reminiscent of the
one in~\cite{cieliebak-mohnke}, where transversality for
holomorphic curves is established using domain-dependent almost
complex structures parametrized by the universal curve over
Deligne--Mumford space. 

%

We conclude this section with a restriction on contributions to the
Tate boundary operator $\p$ arising from transversality. Note that
generators of the Tate chain complex can be written as $u^kx$, where
$k\in\Z$ and $x$ is a either a critical point of $H$ or a critical
point of the Morse function $f$ on a nonconstant $1$-periodic orbit of $H$. 
Since the index of critical points of $\A^\mu$ is nonincreasing,
nonzero matrix elements $\langle x_+,u^\ell x_-\rangle$ can only exist for
$\ell\leq 0$. 

\begin{lemma}\label{lem:vanishing}
By appropriate choice of the Hamiltonian $H$ and a generic $J\in\JJ$
we can achieve that $\langle \p x_+,u^\ell x_-\rangle=0$ in the
following cases:  
\begin{enumerate}
\item[(i)] $x_\pm$ are critical points of $H$ and $\ell<0$; 
\item[(ii)] $x_\pm$ are critical points of the Morse function $f$ on
  distinct nonconstant $1$-periodic orbits $\gamma_\pm$ of $H$ whose
  multiplicities $k_+,k_-$ are relatively prime, and either $\ell<0$,
  or $\ell=0$ and $\CZ(\gamma_+)\leq\CZ(\gamma_-)$.   
\end{enumerate}
\end{lemma}

\begin{proof}
In both cases, we can achieve transversality for all Floer cylinders
in $V$ connecting $x_+$ and $x_-$ (resp.~$\gamma_+$ and $\gamma_-$)
using a generic $z$-independent 
$J\in\JJ(V,\lambda_V)$. In case (i) this holds because, for $H$
sufficiently $C^2$-small on $W$, all Floer cylinders from $x_+$ to
$x_-$ are gradient trajectories, see~\cite{salamon}. In case (ii) it
holds because the relative primeness condition ensures that Floer
cylinders from $\gamma_+$ to $\gamma_-$ cannot be multiply covered. 

For such a $z$-independent $J\in\JJ(V,\lambda_V)$, Tate trajectories
project onto Floer trajectories in $V$. Hence in case (i), $\langle \p
x_+,u^\ell x_-\rangle\neq 0$ implies $\CZ(x_+)\geq \CZ(x_-)$. In
view of the grading condition $\CZ(x_+)=\CZ(x_-)+2\ell+1$, this
yields $\ell\geq 0$. In case (ii), $\langle \p
x_+,u^\ell x_-\rangle\neq 0$ and $\gamma_+\neq\gamma_-$ implies
$\CZ(\gamma_+)\geq \CZ(\gamma_-)+1$. Since the index ${\rm ind}(x_\pm)$
equals $\CZ(\gamma_\pm)$ or $\CZ(\gamma_\pm)+1$, the grading condition 
${\rm ind}(x_+)={\rm ind}(x_-)+2\ell+1$ yields $\CZ(x_+)\leq \CZ(x_-)+2\ell+2$
and thus $\ell\geq 0$. 

This proves vanishing of $\langle \p x_+,u^\ell x_-\rangle$ 
in cases (i) and (ii) for a $z$-independent
$J\in\JJ(V,\lambda_V)$. Now this vanishing persists for a sufficiently
small perturbation of $J$ to a generic element in $\JJ$. 
\end{proof}

\section{Computations and examples}\label{ex}

\subsection{$\C^n$ and subcritical Stein manifolds}

The Tate chain complex for $\C^n$ with integer coefficients is shown
in Figure~\ref{fig:Cn1}. 
\begin{figure}[ht]
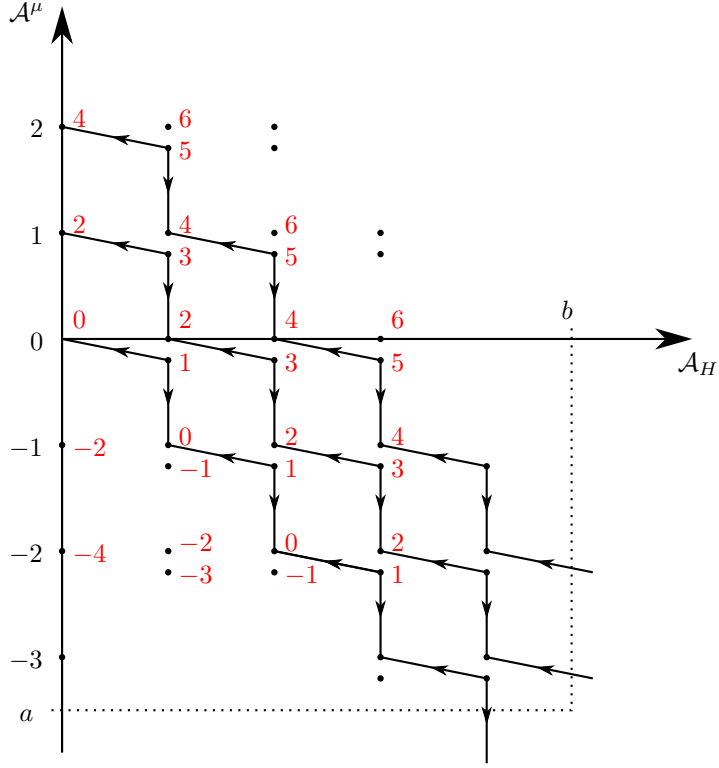
\caption{The Tate chain complex for $\C^n$}\label{fig:Cn1}
\end{figure}
Here the orbits are arranged according to
their $\A_H$- and $\A^\mu$-actions, where each pair of nearby dots
corresponds to the maximum and minimum of a small Morse function on
the circle. The red numbers denote the {\em degree} $|x|:={\rm
  CZ}(x)-n$, i.e., the Conley-Zehnder index minus $n$.  
For $H$ we use a quadratic Hamiltonian close to $z\mapsto|z|^2$,
composed with a strictly convex incrasing function $\R\to\R$. Then 
the generators with $\A^\mu=0$ along the horizontal axis are given by
the minimum $w_0^+$ of $H$ (of degree $0$), and to the generators
$w_k^\pm$ corresponding to the maximum and minimum of a small Morse
function on the periodic orbit $w_k$ of Conley-Zehnder index $n-1+2k$,
of degrees
$$
   |w_k^+|=2k,\qquad |w_k^-|=2k-1. 
$$
The other generators in the diagram are obtained from these by
applying $u^k$ for $k\in\Z$. The horizontal
arrows map generators of odd degree onto those of even degree and thus
compute for each fixed value of $\A^\mu$ the symplectic homology of
$\C^n$, which vanishes. The vertical arrows map the odd generators
$w_k^-$ onto multiples $a_ku^{-1}w_k^+$ of the even generators, so
altogether we have 
$$
   \partial w_k^+=0,\quad \partial w_{k+1}^- =
   w_k^++a_{k+1}u^{-1}w_{k+1}^+,\qquad k=0,1,2,\dots. 
$$  
Here $a_1,a_2,\dots$ is a nondecreasing sequence of natural numbers
which is given by $a_k=k$ for $n=1$, while for higher $n$ each natural
number is repeated $n$ times. The values $a_k$ correspond to the
multiplicities of the closed Reeb orbits on the (slightly perturbed)
unit sphere $S^{2n-1}\subset\C^n$ generating the positive part of
equivariant symplectic homology. 

Note that for each degree there is an infinite sequence of generators
on which $\A_H$ tends to $+\infty$ and $\A^\mu$ to $-\infty$. To
compute the Tate homology groups, it is convenient to introduce the
new generators 
$$
   x_k:=(-1)^ku^{-k}w_k^+,\qquad y_k:=(-1)^{k+1}u^{-k+1}w_k^-
$$
of degrees $|x_k|=0$ and $|y_k|=1$, on which the boundary operator is
given by 
$$
   \partial x_k=0,\quad \partial y_{k+1} = x_k-a_{k+1}x_{k+1},\qquad k=0,1,2,\dots.
$$  
Viewing Figure~\ref{fig:Cn1} as a graph in the obvious way, we see
that it consists of an infinite number of connected components
(``staircases''). The component containing $x_0$ is shown in Figure
\ref{fig:Cn2} for $n=1$ (the case of higher $n$ differs only by the
weights on the vertical arrows), and the other components are obtained
from this one by applying $u^k$ for $k\in\Z$. 
Now we can read off the Tate homology groups. 
\begin{figure}[ht]
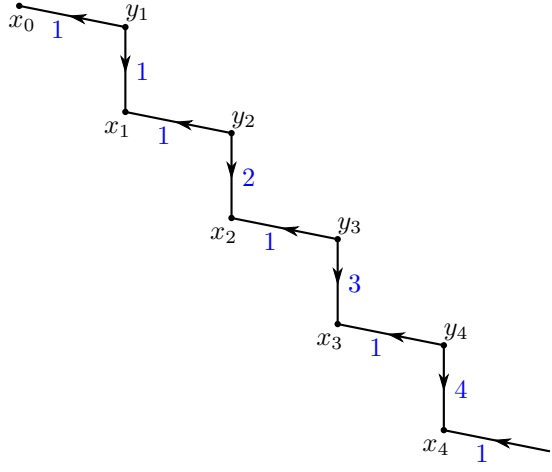
\caption{A connected component in the Tate chain complex for $\C$}\label{fig:Cn2}
\end{figure}
The chain complex for $\overrightarrow{H}\overleftarrow{T}$ consists
of {\em finite} sums of generators. All generators
$x_0,x_1,\dots$ of a given even degree are cycles, and $x_k$ is
homologous to $a_{k+1}x_{k+1}$ for each $k=0,1,2,\dots$. So
$\overrightarrow{H}\overleftarrow{T}=G\otimes_\Z\Z[u,u^{-1}]$, where $G$
is the abelian group described in terms of generators and relations by
$$
   G=\Big\langle x_k, k \in \mathbb{N}\;\Big|\; x_k-a_{k+1} x_{k+1}, k \in
   \mathbb{N} \Big\rangle.
$$
According to Proposition~\ref{prop:Q} below, the group $G$ is
isomorphic to $\Q$. 

By contrast, the chain complex for
$\overleftarrow{H}\overrightarrow{T}$ consists of possibly infinite
sums of generators on which $\A^\mu$ tends to $-\infty$. All such sums
of a given even degree are cycles, but now each even generator $x_k$
is the boundary of the infinite sum $y_{k+1}+\sum_{i=k+1}^\infty
a_{k+1}\cdots a_iy_{i+1}$, so
$\overleftarrow{H}\overrightarrow{T}=0$.  
 
In the complex $HT_a^b$ an even generator is a boundary iff
the staircase that begins there exits the region $\{\A^\mu\geq
a,\,\A_H\leq b\}$ to the bottom, i.e., through the line
$\{\A^\mu=a\}$. It follows that $HT_a^b$ is generated by the elements
of even degree along the right vertical edge of the region $\{\A^\mu\geq
a,\,\A_H\leq b\}$, i.e., whose $\A_H$-value is just below $b$. Note
that their degrees are bounded below by a fixed multiple of $a+b$, see Figure \ref{fig:Cn1}.
Since in the direct limit $b\to\infty$ these degrees go to $\infty$,
we conclude that $\underrightarrow{\lim}HT_a = 0$ for all $a$, and
hence $\underleftarrow{H}\underrightarrow{T} = 0$. 

The inverse limit $\underleftarrow{\lim}HT^k$ as $a\to-\infty$ for
$b=k\in\N$ is the free $\Z[u,u^{-1}]$-module on one generator $x_k$ of
degree $0$, and the map $\underleftarrow{\lim}HT^k\to
\underleftarrow{\lim}HT^{k+1}$ sends $x_k$ onto $a_{k+1}x_{k+1}$. Thus 
$\underleftarrow{\lim}HT^b=\Z[u,u^{-1}]$ for all $b$ and 
$\underrightarrow{H}\underleftarrow{T} = G\otimes_\Z\Z[u,u^{-1}]$, where
$G\cong\Q$ is the group above. So we have shown 

\begin{prop}\label{prop:C}
For the Liouville domain $\C^n$ and integer coefficients, the
diagram~\eqref{maindia2} becomes (with $\Q$ viewed as an abelian group)
\begin{equation}\label{maindia3}
\begin{xy}
 \xymatrix{
\Q\otimes_\Z\Z[u,u^{-1}]
\ar[d]^{Hk}
\ar[r]^{\rho}
& \Q\otimes_\Z\Z[u,u^{-1}]
\ar[d]^{\kappa}\\
0 \ar[r]^{\sigma} &
0.
}
\end{xy}
\end{equation}
\end{prop}

With rational coefficients, the diagram for $\C^n$ becomes (with $Q$
as a field)
\begin{equation}\label{maindia3Q}
\begin{xy}
 \xymatrix{
\Q[u,u^{-1}]
\ar[d]^{Hk}
\ar[r]^{\rho}_\cong
& \Q[u,u^{-1}]
\ar[d]^{\kappa}\\
0 \ar[r]^{\sigma} &
0.
}
\end{xy}
\end{equation}
Note that this is consistent with $\overrightarrow{H}\overleftarrow{T}
\cong H(\C^n;\Q)\otimes_\Q\Q[u,u^{-1}]=\Q[u,u^{-1}]$ being the homology of
the fixed point set, and $\underleftarrow{H}\underrightarrow{T}$ being
the localization of equivariant symplectic homology, which vanishes
for $\C^n$. These two facts together with the vanishing of equivariant
symplectic homology for subcritical Stein manifolds (see
e.g.~\cite{bourgeois-oancea}) and the
isomorphism property of $\rho$ yield more generally 

\begin{prop}
For a Liouville domain $W$ with vanishing equivariant symplectic
homology, e.g.~a subcritical Stein manifold, and rational coefficients
the diagram~\eqref{maindia2} becomes 
\begin{equation}\label{maindia4}
\begin{xy}
 \xymatrix{
H(W;\Q)\otimes_\Q\Q[u,u^{-1}]
\ar[d]^{Hk}
\ar[r]^{\rho}_\cong
& H(W;\Q)\otimes_\Q\Q[u,u^{-1}]
\ar[d]^{\kappa}\\
\overleftarrow{H}\overrightarrow{T}(W;\Q) \ar[r]^{\sigma} &
0.
}
\end{xy}
\end{equation}
\end{prop}

It would be interesting to compute $\overleftarrow{H}\overrightarrow{T}$ in this case. 

\subsection{Cotangent bundles}\label{cotbund}

Let us now restrict to rational coefficients and set
$$
   \Lambda_\Q:=\Q[u,u^{-1}]. 
$$

\begin{prop}\label{prop:cot}
For a cotangent bundle $T^*N$ and rational coefficients, the
diagram~\eqref{maindia2} becomes 
\begin{equation}\label{maindia5}
\begin{xy}
\xymatrix{
H(N;\Q)\otimes_\Q\Lambda_\Q
\ar[d]^{Hk}
\ar[r]^{\rho}_\cong
& \underrightarrow{H}\underleftarrow{T}(T^*N;\Q)
\ar[d]^{\kappa}\\
\overleftarrow{H}\overrightarrow{T}(T^*N;\Q)   \ar[r]^{\sigma} &
\underleftarrow{H}\underrightarrow{T}(T^*N;\Q),
}
\end{xy}
\end{equation}
where $\underleftarrow{H}\underrightarrow{T}(T^*N;\Q)$ depends only on
$\pi_1(N)$ by Goodwillie's theorem~\cite{goodwillie}. In particular,
$\underleftarrow{H}\underrightarrow{T}(T^*N;\Q)=\Lambda_\Q$ if $N$ is simply
connected. \hfill$\square$ 
\end{prop}

{\bf Cotangent bundles of spheres. }For $N=S^n$, the above diagram becomes
\begin{equation}
\begin{xy}
\xymatrix{
\Lambda_\Q\oplus \Lambda_\Q
\ar[d]^{Hk}
\ar[r]^{\rho}_\cong
& \Lambda_\Q\oplus\Lambda_\Q
\ar[d]^{\kappa}\\
\overleftarrow{H}\overrightarrow{T}(T^*S^n;\Q)   \ar[r]^{\sigma} &
\Lambda_\Q\,.
}
\end{xy}
\end{equation}

For $N=S^2$, we can compute all the groups explicitly. 
The Tate chain complex for $T^*S^2$ is shown in Figure~\ref{fig:S2-1},
where the red numbers now denote the Conley-Zehnder index. 
To derive it, we first consider the round metric on $S^2$. For each
$k\in\N$, the $k$-fold covered unparametrized oriented closed
geodesics have Morse index $2k-1$ and form an $S^2$-family. Now we pick 
a non-reversible Finsler perturbation of the round
metric with precisely two simple closed geodesics, the
retrograde (shorter) one $r_1$ and the direct (longer) one $d_1$. For
$k\in\N$, their iterates $r_k$ and $d_k$ correspond to the minimum and
maximum of the corresponding $S^2$-family and thus have Morse index
$2k-1$ and $2k+1$, respectively. 
Each of them gives rise to two generators $r_k^\pm,d_k^\pm$
of the chain complex for the non-equivariant homology of the loop
space $LS^2$, corresponding to the minimum and maximum of a Morse
function on the circle, of Conley-Zenhder (= Morse) indices
$$
   {\rm CZ}(r_k^-)=2k-1,\quad 
   {\rm CZ}(r_k^+)=2k,\quad 
   {\rm CZ}(d_k^-)=2k+1,\quad 
   {\rm CZ}(d_k^+)=2k+2. 
$$
The remaining two generators on the horizontal axis $\A^\mu=0$ arise
from the constant loops corresponding to the minimum and maximum of a
Morse function on $S^2$ and have Conley-Zehnder index $0$ and $2$,
respectively. The other generators in the diagram are obtained from
these by applying $u^k$ for $k\in\Z$. 

\begin{figure}[ht]
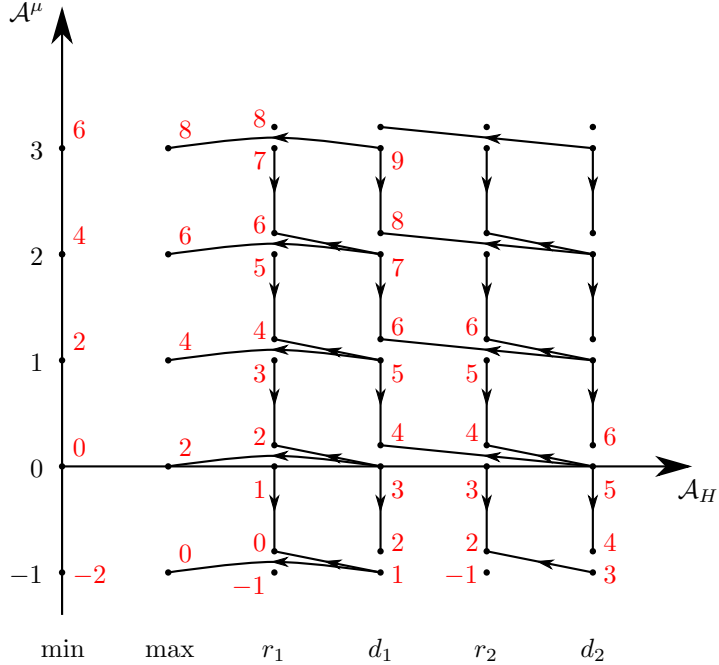
\caption{The Tate chain complex for $T^*S^2$}\label{fig:S2-1}
\end{figure}

The weights of the boundary operator can
be determined as follows. As in the case of $\C^n$, the weights on the
vertical arrows correspond to the multiplicities $k$ of the closed
geodesics $r_k,d_k$. The weights on the short horizontal arrows from
$d_k^-$ to $r_k^+$ must be $2$ for each $k$ in order for the four
generators $r_k^\pm,d_k^\pm$ to compute the homology of $\R P^3$, the
space of parametrized great circles on $S^2$. The matrix elements
$\langle\p r_k^+,d_{k-1}^-\rangle$ and $\langle\p
r_k^-,u^{-1}d_{k-1}^+\rangle$ vanish by Lemma~\ref{lem:vanishing}
because the multiplicities $k,(k-1)$ of $r_k,d_{k-1}$ are relative
prime. So we have
$$
   \p r_k^-=ku^{-1}r_k^+,\quad
   \p r_k^+=0,\quad
   \p d_k^-= 2r_k^+ + c_kd_{k-1}^+ + ku^{-1}d_k^+,\quad
   \p d_k^+=0
$$
for some coefficients $c_k\in\Z$. 

We claim that the {\em $c_k$ must be even}. To see this, we
consider the loop space homology with integer coefficients, which
according to~\cite{cohen-jones-yan} is given by the following quotient
of an exterior algebra on $3$ generators:
$$
   H_{*+2}(LS^2;\Z) \cong \Lambda[x,y,z]/\langle x^2,xz,2xy\rangle,\qquad
   |x|=-2,\ |y|=2,\ |z|=-1. 
$$
From this we read off that
$$
   H_{k}(LS^2;\Z) \cong \begin{cases}
      \Z & k=0, \\
      \Z & k \text{ odd}, \\
      \Z\oplus\Z_2 & k>0 \text{ even}. \\
   \end{cases}
$$
Comparing with Figure \ref{fig:S2-1}, we see that for each $k\geq 1$
the subcomplex consisting of $d_{k-1}^+,r_k^+,d_k^-$ with the
horizontal boundary operators must compute the
homology $\Z\oplus\Z_2$. Since $\p_{\rm hor} d_k^-=2r_k^++c_kd_{k-1}^+$, 
the homology of this subcomplex equals $\Z\oplus\Z_2$ if
$c_k$ is even, and $\Z$ if $c_k$ is odd. Thus $c_k$ must be even. In
Appendix~\ref{app:heat} we use the heat flow to show that the first
weight $c_1$ in fact equals $2$. 

To compute the Tate homology groups, it is convenient to introduce the
new generators 
$$
   x_k:=u^{-k-1}d_k^+,\quad 
   y_k:=u^{-k}d_k^-,\quad
   z_k:=u^{-k+1}r_k^-,\quad
   w_k:=u^{-k}r_k^+ 
$$
of Conley-Zehnder indices ${\rm CZ}(x_k)={\rm CZ}(w_k)=0$ and ${\rm
  CZ}(y_k)={\rm CZ}(z_k)=1$,
on which the boundary operator is given by 
$$
   \p z_k=w_k,\quad
   \p w_k=0,\quad
   \p y_k= 2w_k + c_kx_{k-1} + kx_k,\quad
   \p x_k=0.
$$
As in the case of $\C^n$, we see infinitely many connected components
(related by the action of $u$) on which $\A_H\to\infty$ and $\A^\mu\to
-\infty$. Figure~\ref{fig:S2-2} shows the component containing
$x_0:=u^{-1}\max$ together with the weights of the boundary maps. 

\begin{figure}[ht]
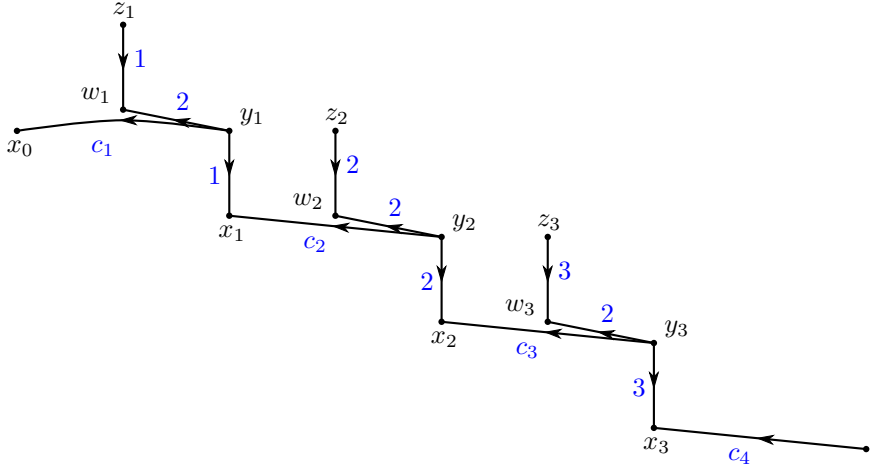
\caption{A connected component in the Tate chain complex for $T^*S^2$}\label{fig:S2-2}
\end{figure}

From Figures \ref{fig:S2-1} and \ref{fig:S2-2} we can now read off the
symplectic Tate homology groups with rational coefficients. For this, observe
that the rational homology of a component as in Figure \ref{fig:S2-2}
does not change if we remove all the generators $w_k,z_k$. The
resulting component looks like the one in Figure \ref{fig:Cn2}, just
with different weights $c_k$ instead of $1$ on the horizontal arrows. 

We claim that all the {\em $c_k$ are nonzero}.
Arguing by contradiction, suppose that $c_k=0$ and $c_i\neq 0$ for
$i<k$, for some $k\geq 1$. Then the finite connected component
containing $x_0$ in Figure~\ref{fig:S2-2} would contribute a direct
summand $\Lambda_\Q$ to
$\underleftarrow{H}\underrightarrow{T}(T^*S^2;\Q)$  
in addition to the one generated by the min, contradicting 
$\underleftarrow{H}\underrightarrow{T}(T^*S^2;\Q)=\Lambda_\Q$ 
from Proposition~\ref{prop:cot}. 

It follows from the claim that, with rational coefficients, each
component has the same contribution to symplectic Tate homology as in
the $\C^n$ case. 
Taking into account the contribution from the min in
Figure \ref{fig:S2-1} (which does not interact with anything else), we
obtain the diagram for $T^*S^2$ with rational coefficients: 
\begin{equation}
\begin{xy}
\xymatrix{
\Lambda_\Q\oplus \Lambda_\Q
\ar[d]^{Hk}
\ar[r]^{\rho}_\cong
& \Lambda_\Q\oplus\Lambda_\Q
\ar[d]^{\kappa}\\
\Lambda_\Q   \ar[r]^{\sigma} &
\Lambda_\Q\,.
}
\end{xy}
\end{equation}

{\bf Cotangent bundles of tori. }
For $N=T^n$ with the flat metric the closed geodesics occur in
Morse-Bott families, one for each class in $H_1(T^n)$, so there are no
gradient trajectories between different components. Since the circle
action is locally free in the nontrivial homology classes, their
contribution to the Tate homology vanishes with rational
coefficients. Thus the only contribution comes from the constants and
we obtain the diagram for $T^*T^n$ with rational coefficients:
\begin{equation}
\begin{xy}
\xymatrix{
H(T^n;\Q)\otimes_\Q\Lambda_\Q
\ar[d]^{Hk}
\ar[r]^{\rho}_\cong
& H(T^n;\Q)\otimes_\Q\Lambda_\Q
\ar[d]^{\kappa}\\
H(T^n;\Q)\otimes_\Q\Lambda_\Q   \ar[r]^{\sigma} &
H(T^n;\Q)\otimes_\Q\Lambda_\Q\,.
}
\end{xy}
\end{equation}

\begin{rem}
By the same argument, all the groups in the diagram coincide also for manifolds of negative sectional curvature, and for products of such manifolds (probably also for manifolds of nonpositive sectional curvature). So in this class of manifolds the Tate homology of a product is given by the tensor product of the Tate cohomologies via the K\"unneth formula. One may wonder whether such a result holds more generally. 
\end{rem}

\begin{ex}\label{ex:kappa}
{\rm We can construct a closed $4$-manifold with the same fundamental
group as $T^n$ as follows. The $n$-fold connected sum
$\#_n(S^1\times S^3)$ has fundamental group the free group on $n$
generators. Surgery on $n\choose2$ embedded loops representing all
commutators of the generators yields a closed $4$-manifold $X_n$
with $\pi_1(X_n)=H_1(X_n)=\Z^n$. By the Mayer-Vietoris sequence, the
second homology of $X_n$ equals $\Z^{n\choose2}\oplus\Z^{n\choose2}$,
where the first $n\choose2$ generators are represented by the
$2$-spheres ${\rm pt}\times S^2$ in the surgery regions $D^2\times S^2$,
and the remaining $n\choose 2$ generators are represented by the
$2$-tori obtained from the disks $D^2\times{\rm pt}$ in the surgery regions
under the identifications on the boundary according to the
commutators.  
So the total homology of $X_n$ has dimension
$d_n=2+2n+2{n\choose2}=n^2+n+2$. Moreover, each $2$-sphere intersects
the corresponding $2$-torus in one point and therefore the
intersection form of $X_n$ is even. Since $X_n$ has no $2$-torsion in
its first homology group, it follows from~\cite[Corollary
  5.7.6]{gompf-stipsicz} that $X_n$ is spin. So we can apply
Theorem~\ref{goodwil} to obtain
$\underleftarrow{H}\underrightarrow{T}(T^*X_n;\Q) = 
\underleftarrow{H}\underrightarrow{T}(T^*T^n;\Q) = \Lambda_\Q^{2^n}$ and we
obtain the diagram for $T^*X_n$ with rational coefficients: 
\begin{equation}
\begin{xy}
\xymatrix{
\Lambda_\Q^{d_n}
\ar[d]^{Hk}
\ar[r]^{\rho}_\cong
& \Lambda_\Q^{d_n}
\ar[d]^{\kappa}\\
\overleftarrow{H}\overrightarrow{T}(T^*X_n;\Q)   \ar[r]^{\sigma} &
\Lambda_\Q^{2^n},
}
\end{xy}
\end{equation}
Since $d_n>2^n$ for $n\leq 4$ and $d_n<2^n$ for $n\geq 6$, this shows
that the map $\kappa$ need neither be injective nor surjective.}
\end{ex}

\appendix

\section{Some presentations of the group $\Q$}

Let $a=(a_k)_{k \in \mathbb{N}}$ be a sequence of positive integers
$a_k$. To such a sequence we associate a group $G_a$, presented as 
$$
   G_a=\Big\langle x_k, k \in \mathbb{N}\;\Big|\; x_k-a_k x_{k+1}, k \in
   \mathbb{N} \Big\rangle.
$$ 
To a sequence $a=(a_k)$ we associate a sequence $a'=(a_k')$ of prime numbers by
removing all the $a_k$ that are equal to $1$, and replacing each $a_k>1$ by
the numbers in its prime decomposition, written in nondecreasing order
and with repetitions. It is easy to see that $G_a=G_{a'}$. Note that
the resulting sequence is finite if only finitely many $a_k$ are
bigger than $1$, and empty if all the $a_k$ are equal to $1$; in these
cases $G_a=G_{a'}$ is obviously isomorphic to $\Z$, generated by the last
$x_k$. The following result shows that for ``generic'' sequences 
the group $G_a$ is isomorphic to $\Q$. 

\begin{prop}\label{prop:Q}
Let $a=(a_k)$ be a sequence of positive integers whose associated
sequence of primes $(a_k')$ satisfies
\begin{equation}\label{infinite}
   \#\{k \in \mathbb{N}: a_k'=p\}=\infty \text{ for every prime number
   } p.
\end{equation}
Then the group $G_a=G_{a'}$ is isomorphic to the additive group $\Q$.
\end{prop}

\textbf{Proof: }
{\bf Step 1. }
By the preceding discussion, we may assume that $a=a'$ is a sequence
of prime numbers. Let us first show that if $a$ and $b$ are two
sequences of prime numbers satisfying~\eqref{infinite}, then the
groups $G_a$ and $G_b$ are isomorphic. 

We denote the generators of $G_a$ by $x_k$ and the generators of $G_b$ by $y_k$.
Our aim is to construct an isomorphism $h \colon G_a \to G_b$ explicitly. For two integers $c$ and $d$ we write $c|d$ if
$c$ divides $d$.  We first define a function 
$m \colon \mathbb{N} \to \mathbb{N}$ as follows
$$m(k)=\left\{\begin{array}{cc}
1 & k=1\\
\min\Big\{\nu \in \mathbb{N}: \prod_{j=1}^{k-1} a_j | \prod_{j=1}^{\nu-1} b_j\Big\} & k \geq 2.
\end{array}\right.$$
Note that by our assumption (\ref{infinite}) the function $m$ is well defined. We next define for every
positive integer $k$ an integer $c_k$ by 
$$c_k=\left\{\begin{array}{cc}
1 & k=1\\
\frac{\prod_{j=1}^{m(k)-1}b_j}{\prod_{j=1}^{k-1}a_j} & k \geq 2.
\end{array}\right.$$
We next define a homomorphism $h \colon G_a \to G_b$ on generators as
$$h(x_k)=c_k y_{m(k)}, \quad k \in \mathbb{N}.$$
It remains to check that $h$ is well defined, i.e.\,it respects the relations in $G_a$. For that purpose we
compute
\begin{eqnarray*}
a_k h(x_{k+1}) &=& a_k c_{k+1} y_{m(k+1)}\\
&=&a_k \frac{\prod_{j=1}^{m(k+1)-1}b_j}{\prod_{j=1}^k a_j} y_{m(k+1)}\\
&=&\frac{\prod_{j=1}^{m(k)-1}b_j}{\prod_{j=1}^{k-1} a_j}
 \prod_{j=m(k)}^{m(k+1)-1} b_j y_{m(k+1)}\\
&=&c_k y_{m(k)}\\
&=&h(x_k).
\end{eqnarray*}
This proves that $h$ is a well defined homomorphism. 

We next check that $h$ is injective. For this purpose we suppose that $\xi \in G_a$ satisfies
$$h(\xi)=0.$$
In view of the structure of the group $G_a$ there exists $d \in \mathbb{Z}$ and a generator $x_k$ such that
$$\xi=dx_k.$$
Hence we obtain
$$0=dh(x_k)=d c_k y_{m(k)}.$$
Since there is no torsion in $G_b$ and $c_k \neq 0$ we conclude that $d=0$ and consequently
$$\xi=0.$$
This proves that $h$ is injective. 

It remains to check that $h$ is surjective. For this purpose choose $\eta \in G_b$. Since the function $m$
is unbounded there exists $k \in \mathbb{N}$ and $e \in \mathbb{Z}$ such that
$$\eta=e y_{m(k)}.$$
In view of our assumption (\ref{infinite}) there exists $\ell>k$ such that
$$c_k| \prod_{j=k}^\ell a_j.$$
We compute
\begin{eqnarray*}
c_k \Bigg( h\bigg(\frac{e \prod_{j=k}^\ell a_j}{c_k} x_{\ell+1}\bigg)-\eta\Bigg)&=&
e h(x_k)-c_k e y_{m(k)}=0,
\end{eqnarray*}
and therefore, since there is no torsion in $G_b$, we obtain
$$ h\bigg(\frac{e \prod_{j=k}^\ell a_j}{c_k} x_{\ell+1}\bigg)=\eta.$$
This proves that $h$ is surjective. In particular, we have shown that
$h$ is a group isomorphism. 

{\bf Step 2. }
It remains to prove $G_a\cong\Q$ for some sequence $(a_k)$ whose
associated sequence of prime numbers satisfies~\eqref{infinite}. 
We choose the sequence $a_k=k+1$, so the group $G_a$ is  
$$
   G_a=\Big\langle x_1,x_2,\dots\;\Big|\;
   x_1-2x_2,\;x_2-3x_3,\;x_3-4x_4,\dots \Big\rangle. 
$$
Note that in $G_a$ we have for all $q<q'$ the identity
$$
   x_q = \frac{q'!}{q!}x_{q'}. 
$$
We claim that the required group isomorphism is given by 
$$
   \phi:\Q\to G_a,\qquad \frac{p}{q}\mapsto p(q-1)!x_q. 
$$
Here $p,q$ need not be relatively prime. Let us first check that
$\phi$ is well-defined. For this we compute for $r\in\N$:
\begin{align*}
   \phi\Bigl(\frac{p}{q}\Bigr)
   &= p(q-1)!x_q 
   = \frac{p(q-1)!(qr)!}{q!}x_{qr}  
   = \frac{p(qr)!}{q}x_{qr} \cr  
   &= pr(qr-1)!x_{qr}
   = \phi\Bigl(\frac{pr}{qr}\Bigr).   
\end{align*}
To show that $\phi$ is a group homomorphism, we compute
\begin{align*}
   \phi\Bigl(\frac{p}{q}\Bigr)+\phi\Bigl(\frac{p'}{q'}\Bigr)
   &= p(q-1)!x_q + p'(q'-1)!x_{q'} \cr
   &= \Bigl(p(q-1)!\frac{(qq')!}{q!} +
   p'(q'-1)!\frac{(qq')!}{q'!}\Bigr) x_{qq'} \cr
   &= \Bigl(\frac{p}{q}+\frac{p'}{q'}\Bigr) (qq')!x_{qq'} \cr
   &= (pq'+p'q)(qq'-1)!x_{qq'} \cr
   &= \phi\Bigl(\frac{p}{q}+\frac{p'}{q'}\Bigr). 
\end{align*}
Injectivity of $\phi$ holds because $rx_q\neq 0$ in $G_a$ for any
$r,q\in\N$,
and surjectivity follows from
$$
   \phi\Bigl(\frac{1}{q!}\Bigr) = (q!-1)!x_{q!} = (q!-1)!\frac{q!}{(q!)!}x_q 
   = x_q. 
$$
This concludes the proof of Proposition~\ref{prop:Q}. 
\hfill $\square$

\section{The heat flow on $S^2$}\label{app:heat}

In view of the results of Salamon and Weber
\cite{salamon-weber, weber}, symplectic homology of a cotangent bundle
can be understood in terms of the heat flow. In this appendix we use
to heat flow on $S^2$ to compute the coefficient $c_1$ in
Figure~\ref{fig:S2-2} to be $c_1=2$. 

We consider $S^2=\{x \in \mathbb{R}^3\;\bigl|\; |x|=1\}$ endowed with its standard round metric, i.e.,
the metric induced from the standard metric on $\mathbb{R}^3$. We abbreviate by
$\mathscr{L}=C^\infty(\mathbb{R}/\mathbb{Z},S^2)$ the free loop space. Closed geodesics on $S^2$ arise as critical points of the energy functional $E \colon \mathscr{L} \to \mathbb{R}$ given for $v \in \mathscr{L}$ by
$$E(v)=\frac{1}{2}\int_0^1 |\dot v|^2 dt.$$
Negative $L^2$-gradient flow lines of $E$ can be interpreted as
solutions $v \in C^\infty(\mathbb{R} \times \mathbb{R}/\mathbb{Z},
S^2)$ of the parabolic PDE 
\begin{equation}\label{grad}
\partial_s v = \nabla_t\p_t v = \partial^2_t v+(\partial_t v)^2 v.
\end{equation}
We consider the geodesic $v_- \in \mathscr{L}$ given by
$$v_-(t)=(0,\cos(2 \pi t),\sin(2 \pi t))$$
and our aim is to determine its unstable manifold $W^u(v_-)$ with
respect to the negative $L^2$-gradient of $E$. We identify $W^u(v_-)$
with all solutions $v$ of (\ref{grad}) satisfying 
\begin{equation}\label{asymp}
\lim_{s \to -\infty} v(s,t)=v_-(t)
\end{equation}
via the map
$$v \mapsto v(0, \cdot) \in \mathscr{L}.$$
For $r \in \mathbb{R}/\mathbb{Z}$ let $R_r$ be the orthogonal $3 \times 3$-matrix
$$R_r=\left(\begin{array}{ccc}
1 & 0 & 0\\
0 & \cos(2\pi r) & -\sin(2\pi r)\\
0 & \sin(2\pi r) & \cos(2 \pi r)\\
\end{array}\right).$$
Consider the following circle action on $\mathscr{L}$, namely for $r \in \mathbb{R}/\mathbb{Z}$
and $v \in \mathscr{L}$ define $r_* v \in \mathscr{L}$ by
$$r_* v(t)=R_r v(t-r).$$
Note that $v_-$ is fixed under this circle action. Moreover, both the energy functional $E$ and the $L^2$-metric are invariant under this circle action and therefore the gradient flow equation (\ref{grad}) is invariant as well. Since the Morse index of $E$ at $v_-$ is 1, we conclude that its unstable manifold is 1-dimensional. Therefore each $v \in W^u(v_-)$ has to be fixed under the $S^1$-action. We conclude that there exist smooth functions $x$ and $y$ depending only on the $s$-variable such that $v$ can be written as
$$v(s,t)=\left(\begin{array}{c}
x(s)\\
\sqrt{1-x(s)^2}\cos\big(2 \pi t -y(s)\big)\\
\sqrt{1-x(s)^2}\sin\big(2 \pi t -y(s)\big)\\
\end{array}\right).$$
Applying (\ref{grad}) to this expression, we arrive at the following system of equations:
$$\left. \begin{array}{c}
\partial_s x=4 \pi^2 (1-x^2)x,\\
\Big(\frac{x \partial_s x}{\sqrt{1-x^2}}-4 \pi^2 x^2 \sqrt{1-x^2}\Big)\cos(2 \pi t-y)=\sqrt{1-x^2} \sin(2\pi t-y)\partial_s y,\\
\Big(\frac{x \partial_s x}{\sqrt{1-x^2}}-4 \pi^2 x^2 \sqrt{1-x^2}\Big)\sin(2 \pi t-y)=-\sqrt{1-x^2} \cos(2\pi t-y)\partial_s y.\\
\end{array}\right\}$$
Plugging the first equation into the latter two equations, we obtain the following equivalent system: 
$$\left. \begin{array}{c}
\partial_s x=4 \pi^2 (1-x^2)x,\\
\partial_s y=0.\\
\end{array}\right\}$$
In particular, $y$ is constant and in view of the asymptotic behaviour (\ref{asymp}) we conclude
$$y=0.$$
The first equation for $x$ is a Bernoulli type ODE whose explicit solution to the initial condition
$$x(0)=x_0 \in (-1,1)$$
is given by the expression
$$x(s)=\frac{x_0}{\sqrt{(1-x_0^2)e^{-8\pi^2 s}+x_0^2}}.$$
Hence $v$ becomes
$$v(s,t)=\left(\begin{array}{c}
\frac{x_0}{\sqrt{(1-x_0^2)e^{-8\pi^2 s}+x_0^2}}\\
\frac{(1-x_0)^2 e^{-8\pi^2 s}}{(1-x_0^2)e^{-8\pi^2 s}+x_0^2}\cos(2 \pi t)\\
\frac{(1-x_0)^2 e^{-8\pi^2 s}}{(1-x_0^2)e^{-8\pi^2 s}+x_0^2}\sin(2 \pi t)\\
\end{array}\right).$$
Note that the positive asymptotic $\lim_{s \to \infty} v(s,t)$ of $v$ is given by the constant loop
$\pm(1,0,0)$, where the sign coincides with the sign of $x_0$. Using
invariance of the gradient flow under the action of the orthogonal
group $O(3)$ on $S^2$, one can now determine explicitly the unstable
manifold for any simple closed geodesic on $S^2$. 

It follows that the fixed positive asymptotic $(1,0,0)$ is hit by precisely two isolated
gradient flow lines. One of these two gradient flow lines has negative
asymptotic $t \mapsto (0,\cos(2\pi t),\sin(2 \pi t))$ 
and the other one has negative asymptotics $t \mapsto (0,\cos(-2\pi
t), \sin(-2\pi t)$. This shows that the coefficient $c_1$ in
Figure~\ref{fig:S2-2} equals $2$, so twice the fundamental class of
$S^2$ becomes zero in localized equivariant
loop space homology.

\end{document}